\documentclass[american,english,graybox,english,11pt]{article}
\usepackage[latin9]{inputenc}
\usepackage{geometry}
\date{}
\usepackage{color}
\usepackage{float}
\usepackage{amsthm}
\usepackage{amsmath}
\usepackage{amssymb}
\usepackage{graphicx}
\usepackage{esint}
\usepackage{subfig}

\makeatletter

\newcommand{\lyxdot}{.}

\numberwithin{equation}{section}
\numberwithin{figure}{section}
\theoremstyle{plain}
\newtheorem{thm}{\protect\theoremname}
  \theoremstyle{plain}
  \newtheorem{algorithm}[thm]{\protect\algorithmname}
  \theoremstyle{plain}
  \newtheorem{prop}[thm]{\protect\propositionname}

\usepackage{algorithm}
\usepackage[T1]{fontenc}

\makeatother

\usepackage[american,english]{babel}
  \addto\captionsamerican{\renewcommand{\algorithmname}{Algorithm}}
  \addto\captionsamerican{\renewcommand{\propositionname}{Proposition}}
  \addto\captionsamerican{\renewcommand{\theoremname}{Theorem}}
  \addto\captionsenglish{\renewcommand{\algorithmname}{Algorithm}}
  \addto\captionsenglish{\renewcommand{\propositionname}{Proposition}}
  \addto\captionsenglish{\renewcommand{\theoremname}{Theorem}}
  \providecommand{\algorithmname}{Algorithm}
  \providecommand{\propositionname}{Proposition}
\providecommand{\theoremname}{Theorem}

\begin{document}

\title{Optimal Sensor Positioning (OSP); A Probability Perspective Study}

\author{Sung Ha Kang, Seong Jun Kim and Haomin Zhou%
\thanks{School of Mathematics, Georgia Institute of Technology, Atlanta, GA
30332, USA email: kang@math.gatech.edu; skim396@math.gatech.edu; hmzhou@math.gatech.edu
\newline
\hspace*{0.2in}This work is partially supported by NSF Awards DMS\textendash{}1042998, DMS\textendash{}1419027, ONR Award N000141310408, and Simons Foundation grant 282311.
}}
\maketitle
\begin{abstract}
We propose a method to optimally position a sensor system, which consists of multiple sensors, each has limited range and viewing angle, and they may fail with a certain failure rate.  
The goal is to find the optimal locations as well as the viewing directions of all the sensors and achieve the maximal surveillance of the known environment. 

We setup the problem using the level set framework. Both the environment and the viewing range of the sensors are represented by level set functions. Then we solve a system of ordinary differential equations (ODEs) to find the optimal viewing directions and locations of all sensors together.   
Furthermore, we use the intermittent diffusion, which converts the ODEs into stochastic differential equations (SDEs), to find the global maximum of the total surveillance area. 
The numerical examples include various failure rates of sensors, different rate of importance of surveillance region, and 3-D setups. They show the effectiveness of the proposed method. 
\end{abstract}

\section{Introduction\label{sec:Introduction}}

We propose a computational method to position a surveillance system with multiple sensors, where each sensor has a limited viewing angle and range.  Some of the sensors may have a non-zero failure rate.  Our goal is to achieve the maximum surveillance of a complicated environment. 
 
%the greatest surveillance area of the region with a given number of multiple sensors which possess finite range and limited coverage angle. Moreover, we consider a scenario when the failure of sensors occurs with probability and then show how to optimize the placements which are robust against these possible failures.

The increasing societal demands for security lead to a growing need for surveillance activities in various environments, such as public places like transportation infrastructures and parking lots, shopping malls and telecommunication relay towers. How to optimally position the sensors is an interesting, important and maybe challenging problem, which can directly impact the public 
safety and security. As reported in  \cite{Murray2007133}, the optimal position of sensors 
may also provide better resources allocation and system performance.

In practice, sensors used in many applications often come with limited sensing ability, for example, 
an object cannot be detected if its distance to the sensor is too far (larger than $r>0$) or it is located outside of the sensors viewing angle ( $[v,v+\theta]$, $0<\theta<2\pi$, where $v$ is a viewing direction and $\theta$ is a limited width of viewing angle). There may exist obstacles that block the view of the sensors. 
In addition, with certain probability, sensors may fail. %Thus, one must consider the hided areas due to the presence of obstacles.

There exist many related studies in optimal sensor positioning. 
%The optimal sensor positioning is a challenging problem.  
Using graph-based approaches, this problem can be posed as a combinatorial optimization problem which has been shown to be NP-hard \cite{KSG2008}. Thus, simple enumeration and search techniques, as well as general purpose algorithms may experience extraordinary difficulty 
in some cases. The placements of sensors have been undertaken for the well-known \emph{art gallery} problem, a classical problem that aims at placing stationary observers to maximize the surveillance area \cite{Aggarwal:1984:AGT:911725} in a cluttered region. The art gallery problem and its variations have drawn considerable attention in recent decades, for example in sensor placements
\cite{959341,Erdem04optimalplacement,Book:Foresti,959342,SISC2}. Similar studies have 
been conducted to other problems such as routes for patrol guards with limited visibility 
or limited mobility. We refer readers to a book \cite{MR921437} and papers
\cite{SISC1,POLSA,She92,Urrutia00artgallery} for more references therein. 
Problems with possible sensor failures have also been studied. For example, 
the use of a Bayesian approach is discussed in \cite{Cameron01101990}, incorporating
environmental data, such as temperatures, from multiple sensors to gather spatial 
statistics can be found in \cite{Book_Cressie}. 
A Gaussian process model is widely used to predict the effect of placing sensors
at particular locations \cite{DCI2002,KSG2008,Worden2001885,ZS_spatial,Zimmerman}. %for taking care of imprecise detections.
See also \cite{FGGH2015} for optimally placing unreliable sensors in a $1$-dimensional environment.

%\cite{Martinez2006661}: Optimal sensor placement and motion coordination for target tracking

Despite of the existence of extensive literature on related topics, the problem that we are 
interested in remains to be challenging.  The main difficulties often come from following aspects: (1) The obstacles have arbitrary shapes. (2) Finding a feasible placement with largest coverage area is often costly because it is an infinite dimensional problem. (3) Finding the globally optimal position, if possible, is usually computationally intractable.  In addition, we consider the case of each sensor having a certain failure rate. 

In this paper, we tackle the problem using the level set formulation together with the intermittent diffusion, a stochastic differential equation (SDE) based global optimization strategy \cite{ID}. The level set framework gives us the freedom to consider various shapes of obstacles and a way to incorporate the limited coverage into the definition of the coverage level set function.
%The total coverage area of multiple sensors can be easily calculated by integrating each sensor's coverage level set function. 
The intermittent diffusion is then used to find one of the optimal viewing directions and locations in all of the feasible positions which gives the largest surveillance area of the environment. % See also our recent work \cite{POLSA}.

The layout of the paper is as follows. In Section 2, the level set formulation of limited sensing range and viewing angle is introduced.  We propose methods to find optimal locations and viewing directions in this setting.  To have a possibility of finding a global optimum, we use the intermittent diffusion. 
In Section 3, we describe the details of having a nonzero failure rate for the sensors and optimal condition for surveilling an area.  Numerical implementation details and different applications are given in Section 4. We conclude by giving a summary and discussion in Section 5.

\section{Optimal positioning of the sensors\label{sec:Coverage-optimization-GD}}

In this section, we  introduce the level set formulation for the coverage function of sensors with limited sensing range and limited width of viewing  angle.  Then we propose a strategy to find the optimal viewing directions of multiple stationary sensors to achieve the largest surveillance of the environment. Finally, we introduce the intermittent diffusion to find the globally optimal placement.

\subsection{Level set formulation for the sensors coverage\label{sub:Level-set-formulation}}

In the level set framework, the environment is described by a level set function, say $\psi(x)$, with positive values outside of the obstacles, negative values inside, and the zero level curves representing boundaries of obstacles.  

A related application on visibility problem uses this level set setting \cite{CT:CMS2005,TCOBS:JCP2004},  where light rays from a vantage point travel in straight lines and are obliterated
upon contact with the surface of an obstacle. A point is seen by the observer (or sensor) placed at a given vantage point, if the line segment between the point and the vantage point does not intersect
any of the obstructions.   In existing work \cite{POLSA,Tsai_robotic_path,LT:2008,LTC:2006} using such a setting, the observer has either a circular coverage (viewed) region, or an infinite range.  
In this paper, we further consider the case of  limited width of viewing angle with a finite range of coverage region.

We denote the computational domain as $D$,
a compact subset of $\mathbb{R}^{2}$. Let $\Omega$ be the collection
of obstacles, i.e. a closed set containing finite number of connected
components comprising one or multiple given obstacles in $D$. The
level set function $\psi(x)$ represents the environment: inside the
obstacles $\Omega$ is negative, and outside is positive. In our algorithms,
we define $\psi(x)$ as the typical signed distance function from the
boundaries of $\Omega$. A sensor located at a point $x$ in $D\backslash\Omega$,
has the line-of-sights as straight line segments originated from $x$
and ended at $y$ with either $|x-y|=r$ and $y\in D\backslash\Omega$,
or $|x-y|\le r$ and $y$ is a point on the boundary of the obstacles.
The coverage region of the sensor is the union of all line-of-sights
emanating from the sensor and the angle to the horizontal axis is
in $[v,v+\theta].$ Figure \ref{fig:A-sensor} depicts the coverage
area of the sensor in the setup. 

\begin{figure}
\begin{centering}
\includegraphics[width=4cm]{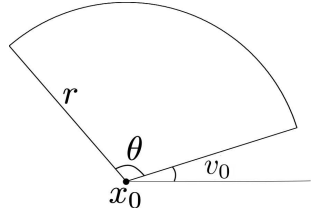}
\par\end{centering}

\caption{\label{fig:A-sensor} A sensor with a finite sensing range and a limited width viewing angle. The coverage area of the sensor is described by a circular sector. $x$ is the location of the sensor, $r$ is the limited sensing range,  $\theta$
is the limited width of viewing angle, and $v$ is the viewing direction.
}
\end{figure}

We use a normalized viewing vector $\vec{n}_{x}(y)=\frac{y-x}{\left|y-x\right|}$
which connects an initial point $x$ with a terminal point $y$. Then,
we define the sensor coverage as a level set function $\phi$ given
by 

\begin{equation}
\phi(y;x,r,v)=\underset{z\in\mathcal{L}(x,y)}{\min}\left\{ \psi(z),r-|z-x|,\varphi(x,y)\right\} \label{eq:visibility-level-set-fcn}
\end{equation}
where $\mathcal{L}(x,y)$ is the line segment connecting $x$ and
$y$, and $\varphi(x,y)=1$ if the angle of $\vec{n}_{x}(y)$ from
the horizontal axis lies in $[v,v+\theta]$ and $\varphi(x,y)=-1$,
otherwise. This way the function $\phi(\cdot;x,r)$ represents the
coverage in a bounded domain $D$ by
\begin{equation}
\begin{cases}
\phi(\cdot;x,r,v)>0, & \mbox{covered regions,}\\
\phi(\cdot;x,r,v)\leq0, & \mbox{non-covered regions.}
\end{cases}\label{eq:visible}
\end{equation}

With the coverage level set function $\phi(y;x,r,v)$ we can calculate
the coverage area of the sensor at $x$ with the viewing direction $v$
by 
\[
V(x,v)=\int_{D}H(\phi(y;x,r,v))dy
\]
where $H$ denotes the one-dimensional Heaviside function.

To compute the coverage function $\phi(y;x,r,v)$, we present a simple
strategy in Algorithm~\ref{PDE-based-algorithm} which follows
the PDE based strategy presented in \cite{TCOBS:JCP2004}. The main
idea is to use the method of characteristics for the nonlinear first
order PDEs, and we modify to handle the finite sensing
range and the limited width of viewing angle. 

\bigskip
\begin{algorithm}
\noindent \textup{\caption{PDE based algorithm to compute coverage}\label{PDE-based-algorithm}}

Input: level set function $\psi$, a sensor $x\in D\backslash\Omega$,
radius $r$, and angles $v,\;\theta$.

1. Set $\phi(y)=\psi(y)$ for $y\in D$.

2. If $y\in B_{r}(x),$ solve for $\phi$:
\begin{equation}
\nabla_{y}\phi(y;x,r,v)\cdot\vec{v}_{x}(y)=0,\label{eq:1st-PDE}
\end{equation}
subject to the boundary condition $\phi=\psi$. 

$\quad$If else, set $\phi(y;x,r,v)=-1$.

3. Update $\phi=\min(\phi,\psi,r-|y-x|,\varphi(x,y))$.
\end{algorithm}
We solve (\ref{eq:1st-PDE}) using the upwind finite differencing
scheme. Note the computational workload in each iteration is dominated
by the second step. However, these steps are only needed for computing
$\phi$ inside the sensor's coverage range, so we can save the computations.

Generalization of the above setting for multiple sensors is immediate.
Let $\{x_{1},\cdots,x_{m}\}$ denote the location of $m$ sensors.
We also let $\{r_{1},\cdots,r_{m}\}$ be the sensing range of the individual sensors,
$\{\theta_{1},\cdots,\theta_{m}\}$ be the limited width of viewing angle, and $\{v_{1},\cdots,v_{m}\}$
be the viewing direction of each, respectively. The coverage of multiple
sensors is the union of the coverage of all sensors. Similar to (\ref{eq:visibility-level-set-fcn}),
we define the coverage level set function with respect to multiple
sensors by
\[
\phi(y;x_{1},\cdots,x_{m},r_{1},\cdots,r_{m},v_{1},\cdots,v_{m})=\underset{i=1,\cdots,m}{\max}\phi(y;x_{i},r_{i},v_{i})
\]
where the associated $\varphi(x_{i},y)$ in $\phi$ is 1 if the angle
of $\vec{n}_{x_{i}}(y)$ from the horizontal axis is in $[v_{i},v_{i}+\theta_{i}]$
and $\varphi(x_{i},y)=-1$, otherwise. We apply Algorithm \ref{PDE-based-algorithm}
to compute $\phi(y;x_{i},r_{i},v_{i})$ with respect to each sensor,
and then take the union to obtain the combined coverage. Moreover,
the coverage area of multiple sensors with respect to the coverage
range $\{r_{1},\cdots,r_{m}\}$ and viewing angles $\{v_{1},\cdots,v_{m}\}$
is given by 
\begin{equation}
V(x_{1},\cdots,x_{m},v_{1},\cdots,v_{m})=\int_{\Omega}H(\phi(y;x_{1},\cdots,x_{m},r_{1},\cdots,r_{m},v_{1},\cdots,v_{m}))dy.\label{eq:joint_visible_volume_function}
\end{equation}
We note that $V$ is Lipschitz continuous under the condition that
there is no fattening in the level set function $\phi$ near the boundary of the obstacles.
%its zero level set.
The details can be found in \cite{CT:CMS2005}.

In the following sections, we focus on maximizing this coverage area with multiple sensors.

% \subsection{Intermittent diffusion?}
 
\subsection{Optimal viewing direction (fixed locations)\label{subsec:Optimal-displacement-stationary}}

We propose a procedure for the optimal position of multiple
stationary sensors to maximize the surveillance area. Our approach incorporates
the previous level set formulation into the optimization methods.
We first consider the case when the location of sensors is fixed and only the viewing direction is allowed to change. The goal is to find the optimal viewing direction of each sensor that maximizes the coverage
area $V$ in (\ref{eq:joint_visible_volume_function}) with fixed
locations $\{x_{1},\cdots,x_{m}\}$. We use the gradient
ascent method, i.e., following the gradient flow to steer the viewing direction:
\begin{equation}
\partial_{t}v_{i}=\nabla_{v_{i}}V(v_{1},\cdots,v_{m}),\; i=1,2,\cdots,m,\label{eq:gd_angle}
\end{equation}
where $\nabla_{v_{i}}$ is the gradient operator with respect to the angle $v_{i}$.

To solve (\ref{eq:gd_angle}) numerically, we use the forward difference
to approximate the time derivative and the central difference to approximate
the gradient operator and, which leads to
\begin{equation}
D_{+}^{k}v_{i}=D_{0}^{h_{v}}V(v_{1},\cdots,v_{m}),\; i=1,2,\cdots,m,\label{eq:gd_central}
\end{equation}
where $k$ is a temporal step size and $h_{v}$ is the step size for
the angle. 
The gradient of $V$ requires
to compute $\phi(y;x_{i},r_{i},v_{i}\pm h_{v})$. We use
the Algorithm \ref{PDE-based-algorithm} for computing $\phi$ over
the grid with equally spaced points with a spacing of spatial step
size $h$ in the computational domain $D$. 
In the evaluation of V, we numerically
integrate $H(\phi)$ using a trapezoidal rule over $D$ and
regularize the Heaviside function by the method proposed in \cite{ETT:2005}:
\[
H_{\epsilon}(\phi)=\begin{cases}
1, & \phi\geq\epsilon,\\
\frac{1}{2}(1+\frac{\phi}{\epsilon}), & |\phi|<\epsilon,\\
0, & \phi\leq-\epsilon,
\end{cases}
\]
using the point-wise scaling with $l_1$ norm of $\nabla\phi$, $\epsilon=\frac{h}{2}\left|\nabla\phi\right|_{l_{1}}$.

The gradient method (\ref{eq:gd_central}) finds a local rather than
a global optimum. To find the globally optimal solution, we employ
the intermittent diffusion method \cite{ID} which adds random perturbations
to the gradient flow (\ref{eq:gd_central}). This helps the deployment
move out of local traps and have a chance to find the globally optimal
one. The following SDE is proposed to be solved before the gradient
ascent flow,
\begin{equation}
dv_{i}(t)=\nabla_{v_{i}}V(v_{1},\cdots,v_{m})dt+\sigma(t)dW(t),\; i=1,\cdots m,
\label{eq:id_sde}
\end{equation}
where $W(t)$ is the Brownian motion in $\mathbb{R}^{m}$, and $\sigma(t)$
is a piecewise constant function alternating between zero and a random
positive constant. The random component perturbs the initial condition
for the gradient ascent so that the solution can escape from the trap
of a local maximizer and approach other ones. 

%We present the steps to approximate the globally optimal solution in Algorithm \ref{alg:Multiple-observer-algorithm}.
%using the intermittent diffusion. 
\begin{algorithm}[H]
\textup{\caption{Optimal viewing direction for globally maximal coverage}\label{alg:Multiple-observer-algorithm}}

\textbf{Input}: sensors $\{x_{1},\cdots,x_{m}\}$, radii
of coverage $\{r_{1},\cdots,r_{m}\}$, viewing directions $\{v_{1},\cdots,v_{m}\}$,
limited widths $\{\theta_{1},\cdots,\theta_{m}\}$, step size for location
$h_x$, step size for angle
$h_{v}$, and temporal step size $k$.

\textbf{Initialize}: Randomly locate the sensors on the allowable area. %boundary of obstacles.

\textbf{Repeat}: Set $j=1$ and iterate $N$ times to obtain $N$
different sets of positioning.

$\quad$$\quad$1. Set $\alpha$ as the scale for diffusion strength,
and $\gamma$ the scale for diffusion time. Let $\sigma:=\alpha d$, 

$\quad$$\quad$and $T:=\gamma t$ where two positive random numbers
$d$, $t$ within $[0,1]$ by uniform distribution.

$\quad$$\quad$2. Set the optimal deployment $L_{opt}=(v_{1},\cdots,v_{m}).$

$\quad$$\quad$3. Taking $L_{opt}$ as the initial condition, compute
the SDE for $t\in[0,T],$ $i=1,\cdots,m.$
\[%\begin{equation}
dv_{i}(t)=\nabla_{v_{i}}V(v_{1},\cdots,v_{m})dt+\sigma(t)dW(t),\; i=1,\cdots m,\label{eq:sde_id}
\]%\end{equation}
$\quad$$\quad$and record the final configuration \textup{as $L_{ini}=(v_{1},\cdots,v_{m}).$}

$\quad$$\quad$4. Compute the gradient ascent flow until convergence
with the initial condition $L_{ini}$ 
\[%\begin{equation}
\partial_{t}v_{i}=\nabla_{v_{i}}V(v_{1},\cdots,v_{m}),\; i=1,2,\cdots,m\label{eq:ga_id}
\]%\end{equation}
$\quad$$\quad$and record the final configuration as $L_{candidate}.$

$\quad$$\quad$5. If $V(L_{candidate})>V(L_{opt})$, set $L_{opt}=L_{candidate}$.

$\quad$$\quad$6. Repeat with $j=j+1.$
\end{algorithm}

One iteration of the procedure is to solve a SDE \eqref{eq:id_sde}
for $t\in[0,T]$ and then compute the gradient ascent flow \eqref{eq:gd_central} until convergence.
This procedure finds one of the globally optimal solutions with probability arbitrarily close to one. 
Indeed, denoting $\alpha\in(0,1)$ by the probability that one iteration attains a close approximation to the globally optimal solution, 
the rate of convergence is proven to be $1-(1-\alpha)^N$ with $N$ iterations. See \cite{ID} for further details.

Figure \ref{fig:multi-stationary-id} illustrates a result of the
algorithm applied to 8 stationary sensors to cover the environment
with polygonal obstacles. For simplicity, we assumed that all sensors have the same sensing range $r$ and the width of viewing angle $\theta$. In
subfigure (a), the obstacles are the shaded three polygons, and the
locations of each sensor are fixed at red crosses. The dotted sector
represents the sensor's coverage if any obstruction is not given by
an obstacle. Otherwise, the red (shaded) area inside the sector shows
the coverage. The objective is now translated into finding the viewing directions such that the union of each coverage is largest. Algorithm
\ref{alg:Multiple-observer-algorithm} seeks one of the desired sets
of the viewing direction $\{v_{1},\cdots,v_{m}\}$ which globally maximizes
(\ref{eq:joint_visible_volume_function}). The final result is presented
in subfigure (b) with $N=80$ iterations. 

\begin{figure}[h]
\begin{centering}
\subfloat[Initial visible volume is 0.9678.]{\begin{centering}
\includegraphics[width=4cm]{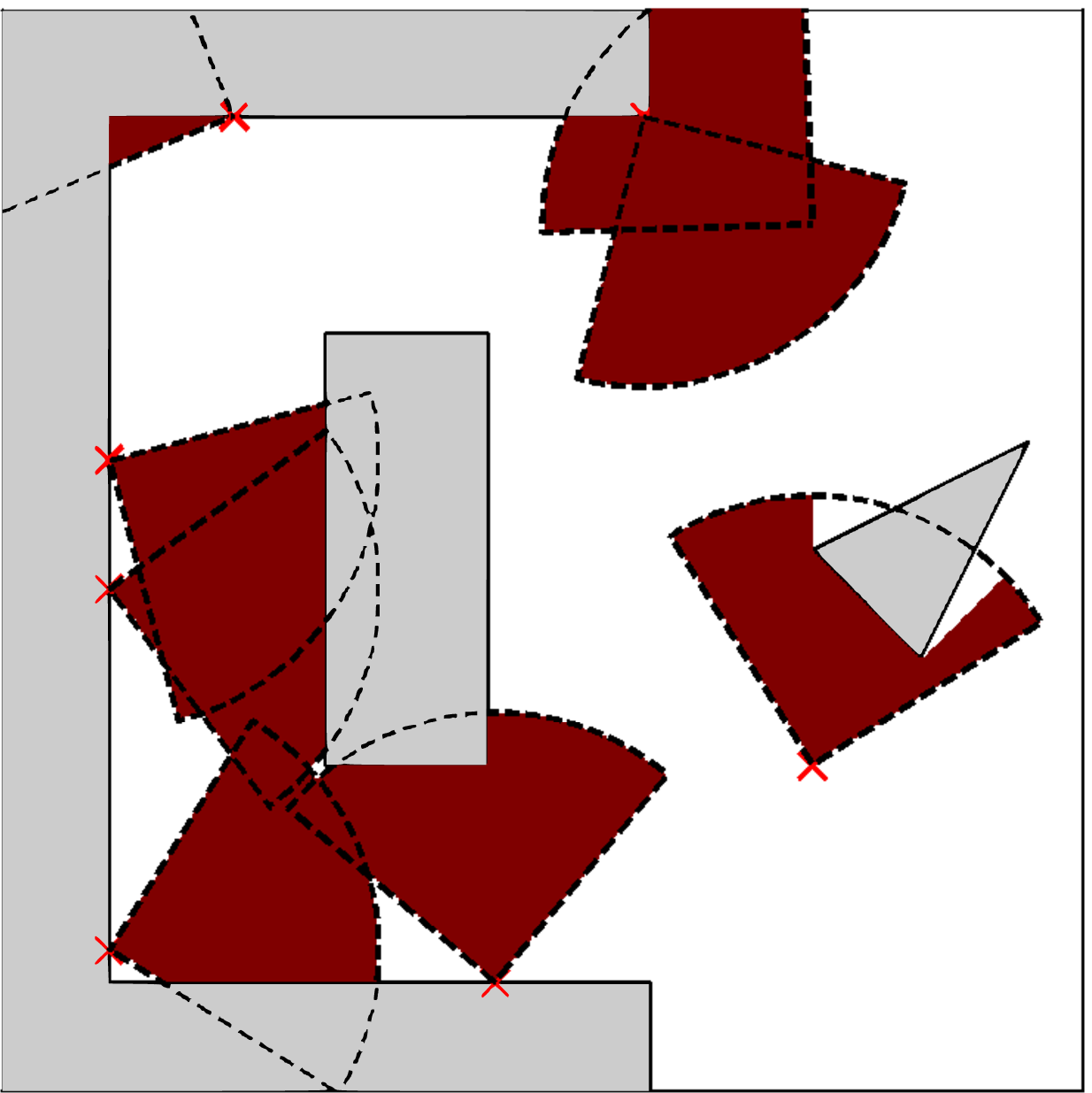}
\par\end{centering}

}\subfloat[Final visible volume is 1.4184.]{\begin{centering}
\includegraphics[width=4cm]{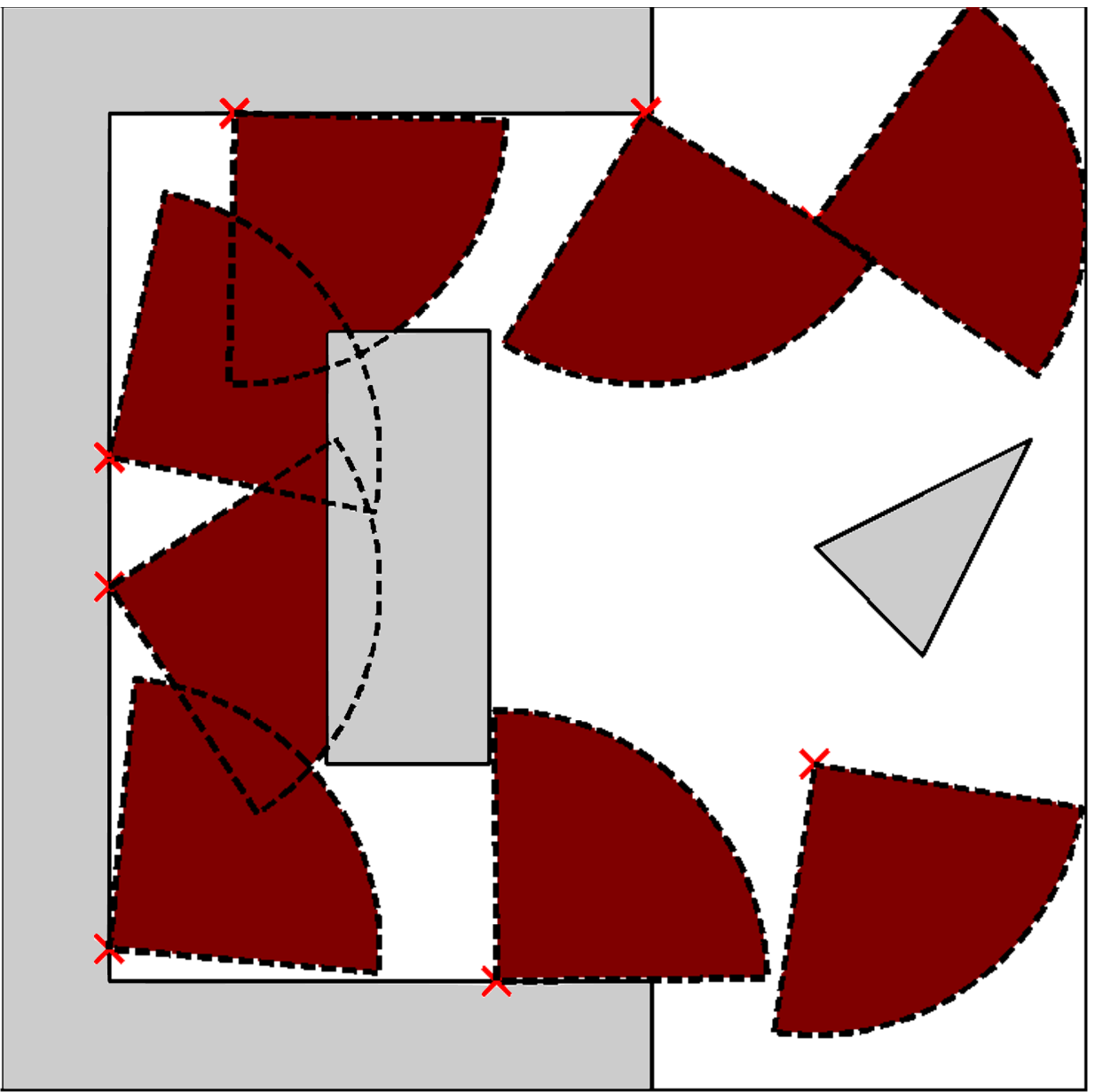}
\par\end{centering}

}\subfloat[]{\begin{centering}
\includegraphics[width=4cm]{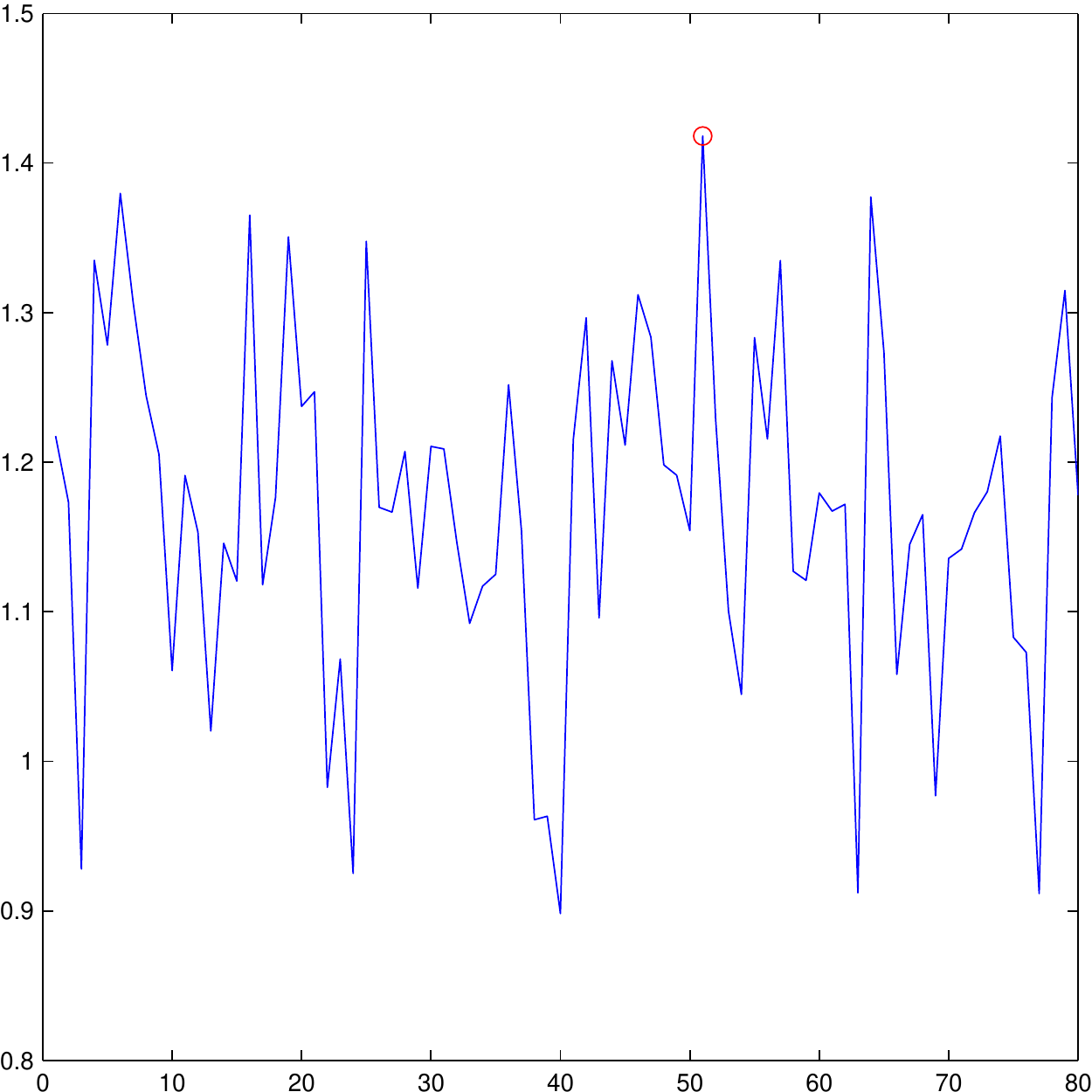}
\par\end{centering}

}
\par\end{centering}

\caption{\label{fig:multi-stationary-id}Optimal viewing directions for multiple sensors
with fixed locations. The viewing directions of multiple sensors
change from the image (a) to the image (b) by Algorithm \ref{alg:Multiple-observer-algorithm}.
(c) The value of coverage area with respect to the iteration. The
maximum coverage is achieved at the 51th iteration. Parameters are
chosen as $r=0.5$, $h=0.005$, $\theta=\pi/2$, $h_{v}=4h$, and
$h_{x}=h$.}
\end{figure}

\subsection{Optimal Sensor Positioning  (OSP)\label{subsec:Optimal-displacement-mobile}}

We now consider the case where one can adjust the location of sensors. %on the boundary of the obstacles. 
The objective is extended to find optimal locations as well as viewing directions that globally maximize the coverage area $V$
in (\ref{eq:joint_visible_volume_function}). The optimal positioning  problem can be reformulated as a $2m$-dimensional problem:
\[
\underset{x_{i},v_{i}\in\mathcal{A}}{\max}V(x_{1},\cdots,x_{m},v_{1},\cdots,v_{m})
\]
where $\mathcal{A}$ is the feasible set of sensors' locations and
viewing directions. Typical locations of the sensors are on the boundary of obstacles $\partial\Omega$. 

Among the sensors, we denote a subset of sensors whose locations are adjustable
by $\{\tilde{x}_{1},\tilde{x}_{2},\cdots,\tilde{x}_{\tilde{m}}\}$,
i.e., we take $\{\tilde{x}_{1},\tilde{x}_{2},\cdots,\tilde{x}_{\tilde{m}}\}$
$\subseteq$ $\{x_{1},\cdots,x_{m}\}$, and each $\tilde{x}_{\tilde{i}}$
is uniquely identified with some sensor $x_{i}$.  Extending the
previous approach for fixed sensors, we employ the intermittent
diffusion and the gradient ascent to maximize the coverage area. However,
due to the adjustability of some sensors, a parametrization of the
boundary of obstacles is required to conduct the optimization with
respect to the location. This can be
done by using a reinitialization method given in \cite{TO:Acta-Numerica:2005} and parameterizing zero level curves which represent boundaries
of obstacles.

Now, the gradient ascent is used for the modified equation of (\ref{eq:gd_angle}):
\begin{eqnarray}
\partial_{t}v_{i} & = & \nabla_{v_{i}}V(x_{1},\cdots,x_{m},v_{1},\cdots,v_{m}),\; i=1,2,\cdots,m,\nonumber \\
\partial_{t}\tilde{x}_{j} & = & \nabla_{\tilde{x}_{j}}V(x_{1},\cdots,x_{m},v_{1},\cdots,v_{m}),\; j=1,2,\cdots,\tilde{m},\label{eq:gd_mobile}
\end{eqnarray}
where $\nabla_{v_{i}}$ is the gradient operator with respect to the viewing direction $v_{i}$,
and $\nabla_{\tilde{x}_{j}}$ is the gradient operator with respect
to the location $\tilde{x}_{j}$ along the boundary of obstacle where the sensor
$\tilde{x}_{j}$ lies in. 

As in (\ref{eq:gd_angle}), when the solution reaches a steady state, this gives one of many local optimal solutions.  To find a globally optimal solution, we use the intermittent diffusion. It adds random perturbations intermittently to the gradient flow (\ref{eq:gd_mobile})
which leads to the following SDEs, 
\begin{eqnarray}
\partial_{t}v_{i} & = & \nabla_{v_{i}}V(x_{1},\cdots,x_{m},v_{1},\cdots,v_{m})dt+\sigma_{1}(t)dW_{1}(t),\; i=1,2,\cdots,m,\nonumber \\
\partial_{t}\tilde{x}_{j} & = & \nabla_{\tilde{x}_{j}}V(x_{1},\cdots,x_{m},v_{1},\cdots,v_{m})dt+\sigma_{2}(t)dW_{2}(t),\; j=1,2,\cdots,\tilde{m},\label{eq:ID_loca-1}
\end{eqnarray}
where $W_{1}(t)$ and $W_{2}(t)$ are the Brownian motions in $\mathbb{R}^{m}$
and $\mathbb{R}^{\tilde{m}}$, respectively. In addition, $\sigma_{1}(t)$
and $\sigma_{2}(t)$ are piecewise constant functions alternating
between zero and a random positive constant. This set of SDEs is followed
by the gradient ascent which can escape from the trap of a local minimizer
and approach other ones. 

%Now, we are ready to present the algorithm to compute the globally optimal deployment of multiple sensors in 
Algorithm \ref{alg:Multiple-movable observer-algorithm} presents Position Optimization for Sensors (POS) for maximal coverage area. Here we use the term positioning to describe both locations and  viewing directions.

\begin{algorithm}
\textup{\caption{OSP}\label{alg:Multiple-movable observer-algorithm}}

\textbf{Input}: sensors $\{x_{1},\cdots,x_{m}\}$, radii
of coverage $\{r_{1},\cdots,r_{m}\}$, viewing directions $\{v_{1},\cdots,v_{m}\}$,
limited widths $\{\theta_{1},\cdots,\theta_{m}\}$, step size for location
$h_{x}$, step size for angle $h_{v}$, and temporal step size
$k$. Also, adjustable sensors\textup{ $\{\tilde{x}_{1},\tilde{x}_{2},\cdots,\tilde{x}_{\tilde{m}}\}\subseteq\{x_{1},\cdots,x_{m}\}.$}

\textbf{Initialize}: Randomly locate the sensors on the allowable area. %Randomly locate adjustable sensors on the boundary of obstacles.

\textbf{Repeat}: Set $j=1$ and iterate $N$ times to obtain $N$
different deployments. 

$\quad$$\quad$1. Set $\alpha$ as the scale for diffusion strength,
and $\gamma$ the scale for diffusion time. Let $\sigma:=\alpha d$, 
and $T:=\gamma t$ where two positive random numbers
$d$, $t$ within $[0,1]$ by uniform distribution.

$\quad$$\quad$2. Set the optimal deployment $L_{opt}=(x_{1},\cdots,x_{m},v_{1},\cdots,v_{m}).$

$\quad$$\quad$3. Taking $L_{opt}$ as the initial condition, compute
the SDE (\ref{eq:ID_loca-1}) for $t\in[0,T]$ and record the 
final configuration \textup{as $L_{ini}=(x_{1},\cdots,x_{m},v_{1},\cdots,v_{m}).$}

$\quad$$\quad$4. Compute the gradient ascent flow (\ref{eq:gd_mobile})
until convergence with the initial condition $L_{ini}$ and record the final configuration as $L_{candidate}.$

$\quad$$\quad$5. If $V(L_{candidate})>V(L_{opt})$, set $L_{opt}=L_{candidate}$.

$\quad$$\quad$6. Repeat with $j=j+1.$
\end{algorithm}

Evaluations of $\nabla_{0}^{h_{v}}V$ and $\nabla_{0}^{v_{x}}V$
require to compute $\phi(y;x_{i},r_{i},v_{i}\pm h_{v})$ for 
$i=1,\cdots,m$ and $\phi(y;\tilde{x}_{j}\pm h_{x},r_{i},v_{i})$ for $i=1,\cdots,\tilde{m}$ if $x_{j}$
is allowed to move along the boundary. In Step 4, we use Euler's method
to update $v_{i}$ by $v_{i}+kD_{0}^{h_{v}}V(x_{1},\cdots,x_{m},v_{1},\cdots,v_{i},\cdots,v_{m})$
for $i=1,\cdots,m$ and $\tilde{x}_{j}$ by $\tilde{x}_{j}+kD_{0}^{h_{x}}V(x_{1},\cdots,\tilde{x}_{j},\cdots,x_{m},v_{1},\cdots,v_{m})$
for movable $\tilde{x}_{j}$, respectively. Figure \ref{fig:multi-mobile}
describes a result of Algorithm \ref{alg:Multiple-movable observer-algorithm}
applied to the same environment in Figure \ref{fig:multi-stationary-id},
but all of the sensors
are movable except for the two sensors not on the boundary of any obstacle. An optimal position which locally maximizes (\ref{eq:joint_visible_volume_function})
is presented in subfigure (b). It is evident that the resulting coverage
area is much larger than the one in Figure \ref{fig:multi-stationary-id}
because mobile sensors have more degrees of freedom than that of stationary
sensors. For simplicity, we suppose that all sensors have the same
sensing range $r$ and width of viewing angle $\theta$.

\begin{figure}
\begin{centering}
\subfloat[Initial visible volume is 0.9678.]{\begin{centering}
\includegraphics[width=4cm]{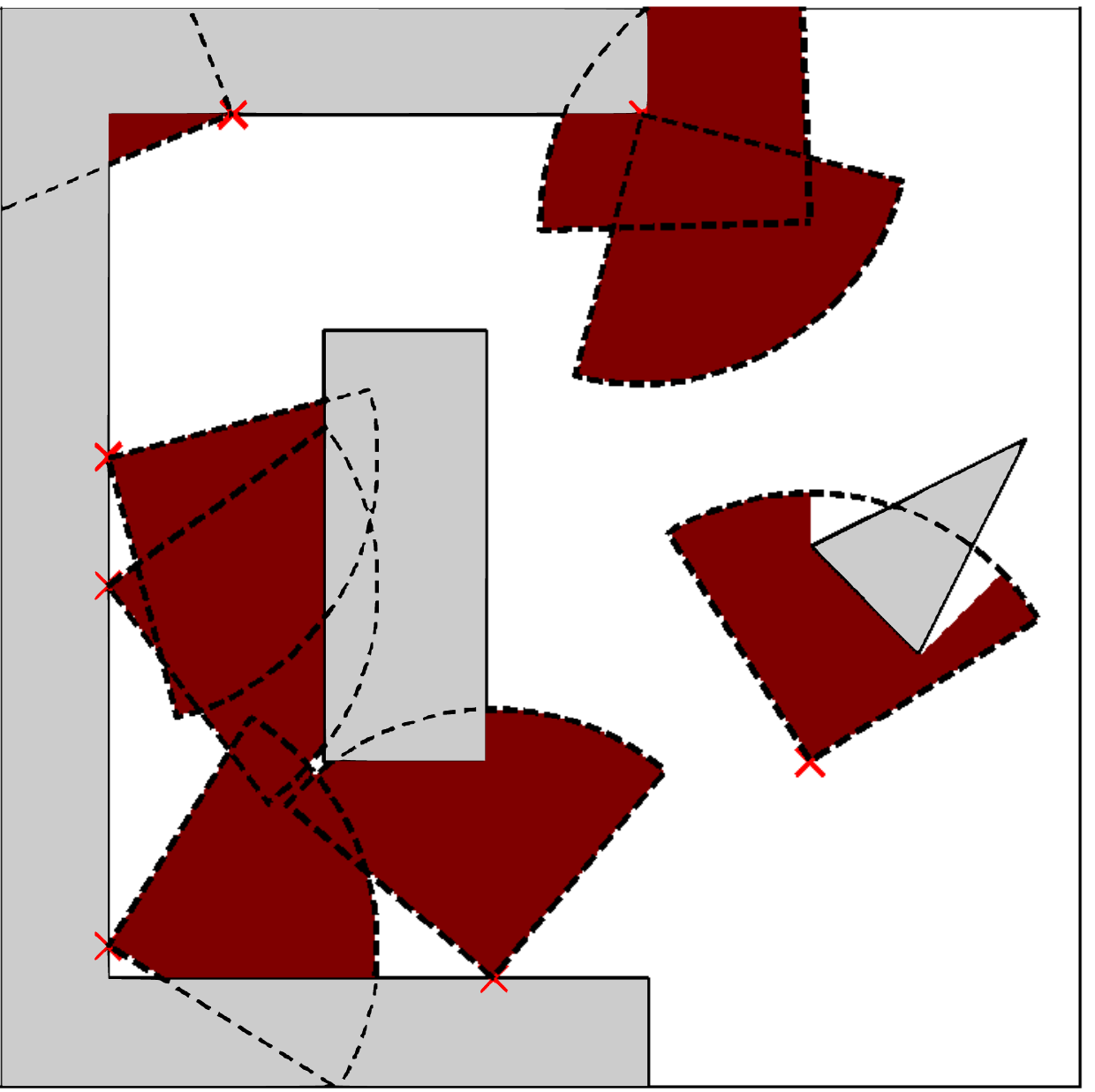}
\par\end{centering}

}\subfloat[Final visible volume is 1.5461.]{\begin{centering}
\includegraphics[width=4cm]{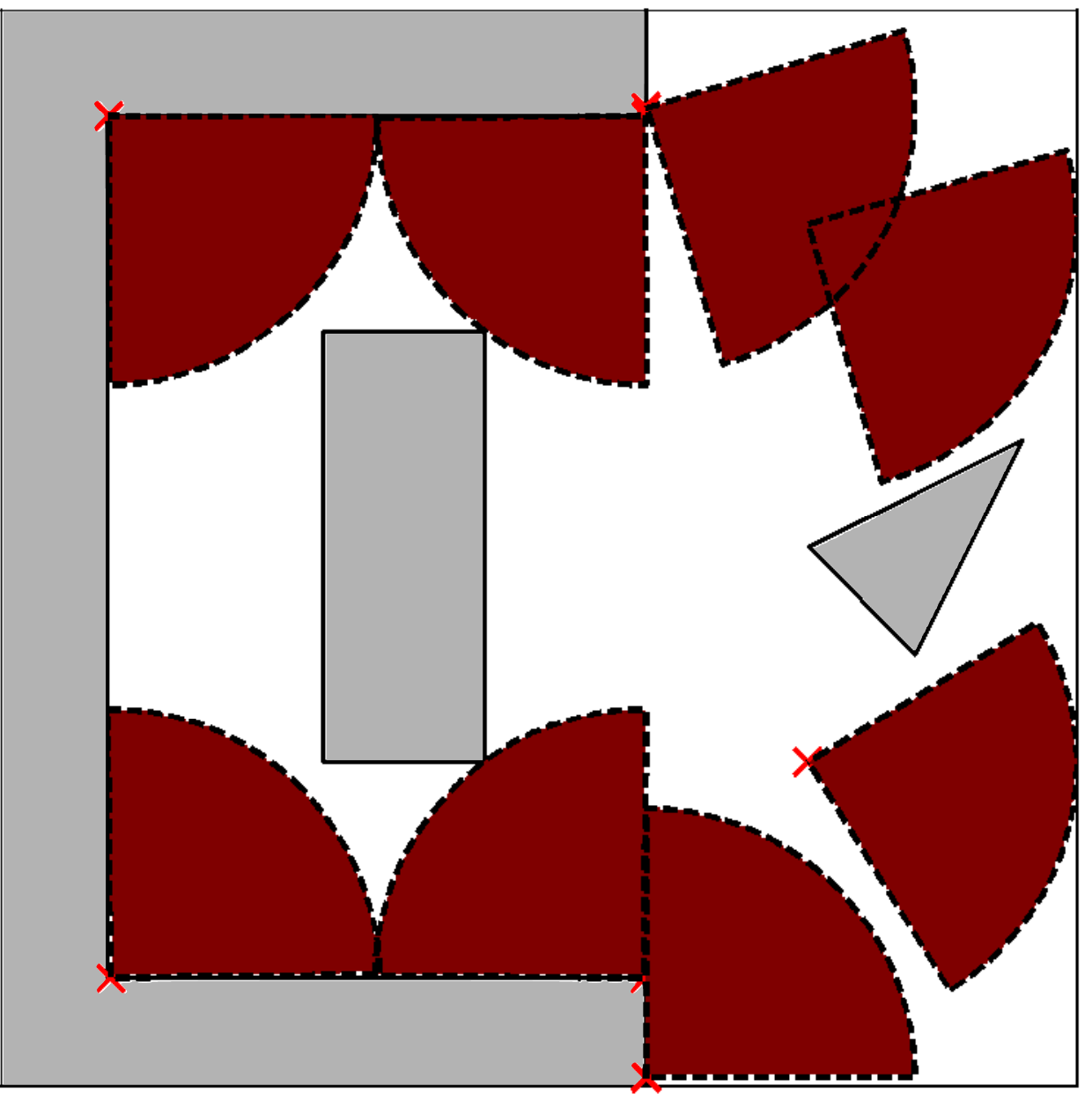}
\par\end{centering}

}\subfloat[]{\begin{centering}
\includegraphics[width=4cm]{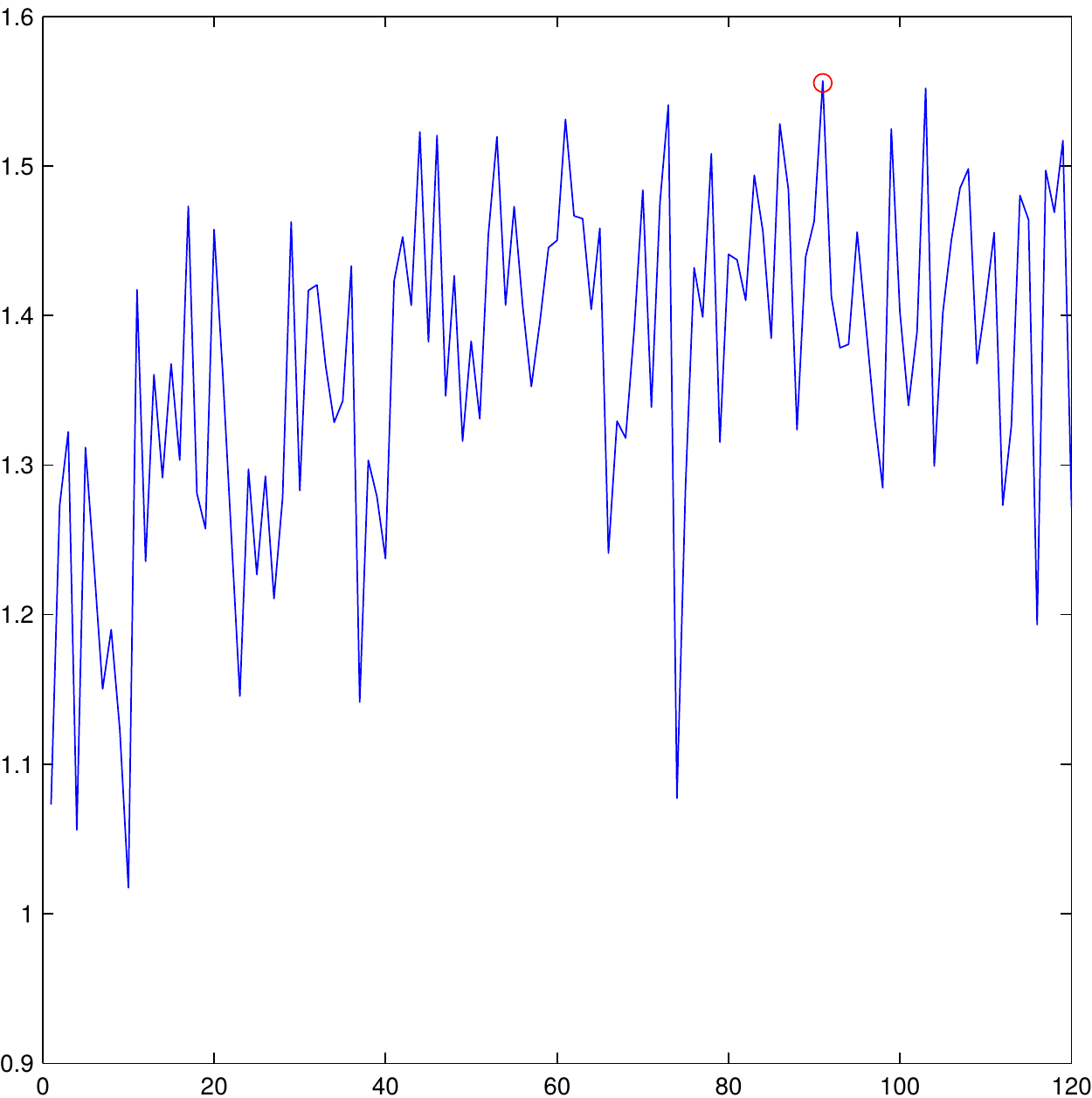}
\par\end{centering}

}
\par\end{centering}

\caption{\label{fig:multi-mobile}
Optimal locations and viewing directions of multiple sensors.
The locations and viewing directions of multiple sensors change from the
image (a) to the image (b) by Algorithm \ref{alg:Multiple-movable observer-algorithm}. (c) The
value of coverage area with respect to the iteration. The maximum
coverage is achieved at the 91st iteration. Parameters are chosen
as $r=0.5$, $h=0.005$, $\theta=\pi/2$, $h_{v}=4h$, and $h_{x}=h$.}
\end{figure}

\section{Optimal sensor positioning with a failure rate\label{sec:possible_failure}}

The failure of sensors is crucial for
the surveillance system to achieve its intended task, and 
this can be added to the stage of sensor positioning. 
We extend the previous approach to find optimal positions when the sensors have a non-zero probability of failure. 
As a priori estimate, we consider an expected coverage area of the sensors in the objective functional and analyze an optimal solution to maximize it. If there are sufficiently many sensors in the sense that the collective coverage area is much larger than the area of $D\backslash\Omega$, the optimal solution will be shown to uniformly spread the sensors over the monitoring area.
%We also propose to find the optimal positioning of multiple sensors which maximizes the expected coverage
%area of the environment.

We assume that sensors fail independently, i.e., if two sensors
have a failure probability $p_{1}$ and $p_{2}$, respectively, then
$p_{1}p_{2}$ is the joint failure probability. Given a point $y$
in the environment, we denote $0<p(y)<1$ by a probability that a
set of sensors fails to detect $y$. If $y$ is not in any
of the sensor's coverage, then $p(y)$ is set to be 1. Furthermore,
we denote $p_{k}(y)$ by a probability that the individual $k$-th
sensor fails to detect a point $y$. 

The expectation of the coverage area with possible sensor failures is defined as
\begin{equation}
\mbox{\ensuremath{\mathbb{E}}}(V(x_{1},\cdots,x_{m},v_{1},\cdots,v_{m}))=\int_{D}H(\phi(y;x_{1},\cdots,x_{m},r_{1},\cdots,r_{m},v_{1},\cdots,v_{m}))(1-p(y))dy.\label{eq:expected_coverage}
\end{equation}
That is, the functional is the integral of the products of covered area with the probabilities that the sensors to cover it are working. 
Note that if there is no sensor failure, then $p(y)=0$ so that the
objective functional (\ref{eq:expected_coverage}) is reduced to (\ref{eq:joint_visible_volume_function}).
The optimization problem is to maximize the expected area of coverage
which is formulated as: 
\begin{equation}
\underset{x_i, v_i}{\max} \;\mbox{\ensuremath{\mathbb{E}}}(V(x_{1},\cdots,x_{m},v_{1},\cdots,v_{m}))\label{eq:objective_functional_random}
\end{equation}
subject to the locations of the sensors adjustable on either the boundary of obstacles
$\partial\Omega$ or the boundary of the computational domain.

For simplicity, we discretize a two dimensional environment by uniformly placing a grid $(ih,jh)$
on it and denote $0<p_k (i,j)<1$ by a probability that the $k$-th sensor fails
to detect a grid point $(ih,jh)$. 
%$p_{k}$, $k=1,2,\cdots,n$, respectively. 
With an uniformly space
grid, the expected area of sensors' coverage is numerically approximated
by 
\[
\underset{i,j}{\sum}h^{2}\left(1-\overset{n}{\underset{k=1}{\prod}}p_{k}(i,j)\right).
\]

The following considers how the multiple sensors should be positioned in order to
maximize the objective functional \eqref{eq:expected_coverage}.  We
show that positioning multiple sensors with less overlapping area is closer to the
optimal solution than doing so with more overlapping area.

\begin{prop}
Positioning multiple sensors with minimal overlapping regions attains the greater 
expected coverage area than those with 
more overlapping regions.
\end{prop}
\begin{proof}
%Suppose that there are $m$ sensors with an identical coverage area $A$ and a failure probability.
Suppose that the small region $\mathcal{R}$ with an area $B$ is
given, and that this region is fully covered by $m$ sensors. Firstly,
suppose that their coverages do not intersect. 
We denote $\mathcal{R}_{nonoverlap}$ by the sensor positions where all sensors do not intersect and cover 
entire $\mathcal{R}$. Then the contribution
of this sensor positioning to the total expected coverage area is 
\begin{equation}
\mathbb{E}[\mathcal{R}_{nonoverlap}]=B\left(\overset{m}{\underset{k=1}{\sum}}(1-p_{k})\right)
\end{equation}
where the $k$-th sensor has a failure probability $p_{k}$. 

Secondly, suppose that several $\tilde{m}$ $(\leq m)$ sensors overlap and fully
cover $\mathcal{R}$. In this case, the contribution to the total
expected coverage area is 
\begin{eqnarray}
\mathbb{E}[\mathcal{R}_{overlap}] & = & B\left(1-\overset{\tilde{m}}{\underset{k=1}{\prod}}p_{k}\right).
\end{eqnarray}
Now, we compare $\overset{m}{\underset{k=1}{\sum}}(1-p_{k})$ with
$1-\overset{\tilde{m}}{\underset{k=1}{\prod}}p_{k}$. Since the failure
probability is between 0 and~1, 
\[
1-\overset{m}{\underset{k=1}{\prod}}p_{k}>1-\overset{\tilde{m}}{\underset{k=1}{\prod}}p_{k}.
\]
 Using the mathematical induction, suppose that 
\[
\overset{m}{\underset{k=1}{\sum}}(1-p_{k})\geq1-\overset{m}{\underset{k=1}{\prod}}p_{k}
\]
holds true. Then, multiplying $p_{n+1}$ yields that 
\[
p_{n+1}\overset{n}{\underset{k=1}{\sum}}(1-p_{k})\geq p_{n+1}-\overset{n+1}{\underset{k=1}{\prod}}p_{k}.
\]
Now,
\begin{eqnarray*}
\overset{n+1}{\underset{k=1}{\sum}}(1-p_{k}) & \geq & \overset{n+1}{\underset{k=1}{\sum}}(1-p_{k})-(1-p_{n+1})\overset{n}{\underset{k=1}{\sum}}(1-p_{k})\\
 & = & 1-p_{n+1}+p_{n+1}\overset{n}{\underset{k=1}{\sum}}(1-p_{k})\\
 & \geq & 1-\overset{n+1}{\underset{k=1}{\prod}}p_{k}.
\end{eqnarray*}
Therefore, $\mathbb{E}[\mathcal{R}_{nonoverlap}]\geq\mathbb{E}[\mathcal{R}_{overlap}]$.
\end{proof}
\noindent 
In other words, when the sensor fails with non-zero probability, spreading multiple
sensors over the monitoring area as much as possible attains the largest
expectation. 
Consequently, if the algorithm maximizes \eqref{eq:expected_coverage},
the optimal position will not only cover all of the free areas,
$D\backslash\Omega$, but also minimize the overlaps among the sensors.
We compare two scenarios in the following section; optimal positioning of the sensors with or without a failure rate
and test the performance of our proposed algorithm.

\section{Numerical applications\label{sec:more_examples}}

In this section, various numerical experiments are given. We use Algorithm~\ref{alg:Multiple-movable observer-algorithm} with the functional \eqref{eq:expected_coverage} for most of the examples, considering a nonzero failure rate.  In one example depicted in Figure \ref{fig:no_fail} where the failure rate $p=0$, we use the  objective functional (\ref{eq:joint_visible_volume_function}).

%First, two scenarios where all of the sensors work with or without a failure probability are considered. Afterwards, we always assume that each sensor has the same probability of a possible failure. We then consider a problem of optimal deployment in the environment with a corner. Supporting a robust entrance monitoring for a certain area is also considered. Finally, we extend the above approach to a 3D environment and show an experiment. 

\subsection{Comparison: failure rate $p=0$ and $p\neq 0$}
Two examples compare the sensor positioning with or without a failure rate. 
We first consider a problem with sensors which work perfectly, i.e., the
probability of failure is $p=0$. With a set of these sensors, a rectangular
environment without obstacles will be monitored. Suppose that the
number of sensors is large so that the environment can be fully covered
with some redundancies. Figure \ref{fig:no_fail}(a) shows an initial
position of 16 sensors which is randomly determined by putting 8
sensors for each side. We display the number of overlaps in the sensors
coverage by color scales using a horizontal color bar. The area of the
region to be covered is $1$. With the finite sensing range $r=0.6$
and the limited width of sensing angle $\theta=\pi/3$, the collective sensors\rq{}s
coverage area, if they are disjoint, is 3.016.

As shown in subfigures (b) and (c), with iterations our algorithm
focuses on covering all of the regions without any empty space to
achieve the optimal solution. The number of coverage overlaps is not
considered.

\begin{figure}
\begin{centering}
\hspace*{-1.0cm}\includegraphics[width=4cm]{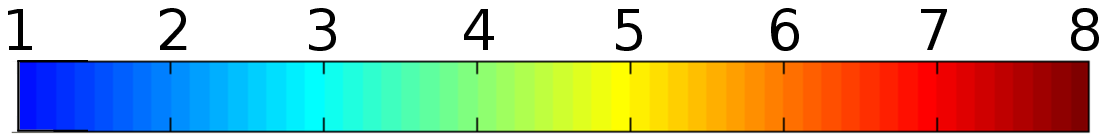}
\par\end{centering}

\begin{centering}
\hspace*{-0.5cm}\subfloat[Initial visible area is 0.4131.]{\begin{centering}
\includegraphics[width=4cm]{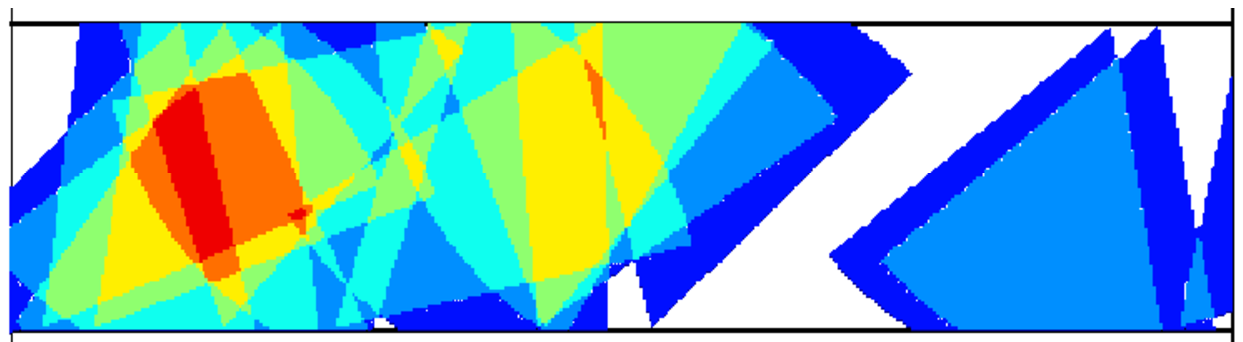}
\par\end{centering}

}\subfloat[After 10 iterations, resulting visible area is 0.9928.]{\begin{centering}
\includegraphics[width=4cm]{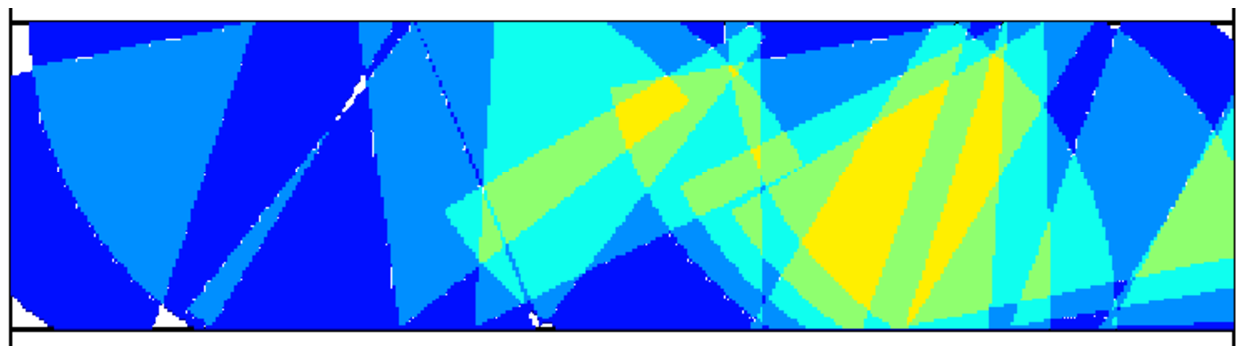}
\par\end{centering}

}\subfloat[After 50 iterations, resulting visible area is 0.9996.]{\begin{centering}
\includegraphics[width=4cm]{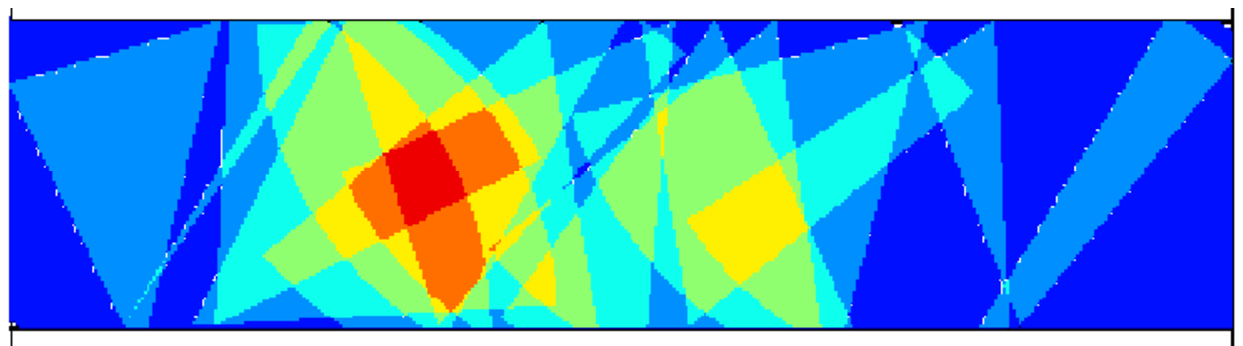}
\par\end{centering}

}
\par\end{centering}

\caption{\label{fig:no_fail} Position optimization with no possible failure.
The resulting sensor positions are obtained using the parameters
$r=0.6$, $h=0.005$, $\theta=\pi/3$. $N=50$ iterations of our algorithm are used.
The horizontal colorbar displays how many sensors cover the area using the colormap.}
\end{figure}

In real applications, sensors may fail. 
%It is worth pointing out that practical sensors can fail in applications.
Assuming that the sensor fails to work with a probability $p$, our algorithm
maximizing (\ref{eq:expected_coverage}) not only focuses on covering
all of the free areas but also minimizes the overlaps among the sensors
as proven in Proposition~1. We take the same environment of Figure~\ref{fig:no_fail},  but assume that each sensor has a 50\% chance of failure,
i.e., with a probability $p=0.5.$ Figure \ref{fig:can_fail} shows
the results of our algorithm. Since we compute the expected area of
covered region with a probability, the initial value of the objective
functional is different from the one in Figure \ref{fig:no_fail}.
The algorithm automatically finds the optimal deployment of sensors
which leads to the full coverage as well as the minimal area of the
intersections. The subfigure (c) shows evenly distributed coverages
of sensors compared to Figure \ref{fig:no_fail}(c).

\begin{figure}[H]
\begin{centering}
\hspace*{-1.0cm}\includegraphics[width=4cm]{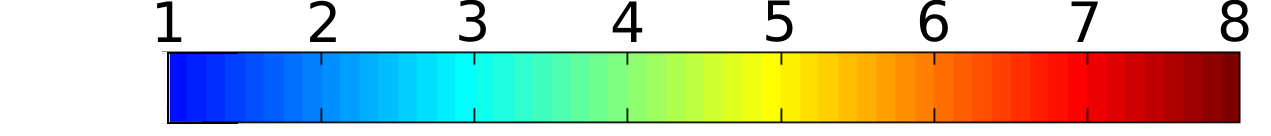}
\par\end{centering}

\begin{centering}
\hspace*{-0.5cm}\subfloat[Initial expected area is 0.3130.]{\begin{centering}
\includegraphics[width=4cm]{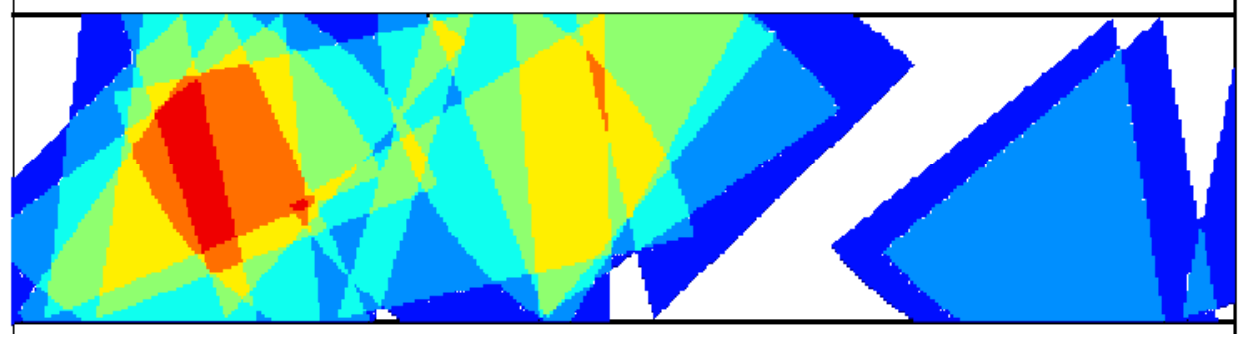}
\par\end{centering}

}\subfloat[After 10 iterations, resulting expected visible area is 0.7810.]{\begin{centering}
\includegraphics[width=4cm]{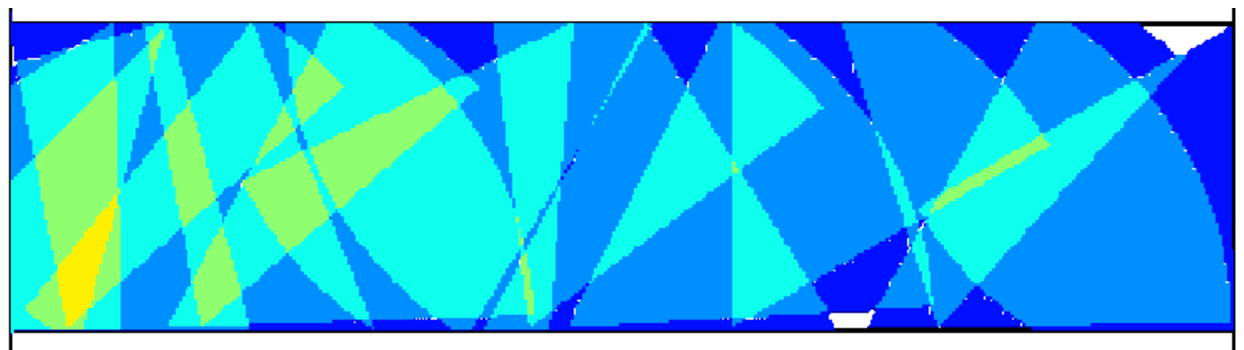}
\par\end{centering}

}\subfloat[After 50 iterations, resulting expected visible area is 0.8079.]{\begin{centering}
\includegraphics[width=4cm]{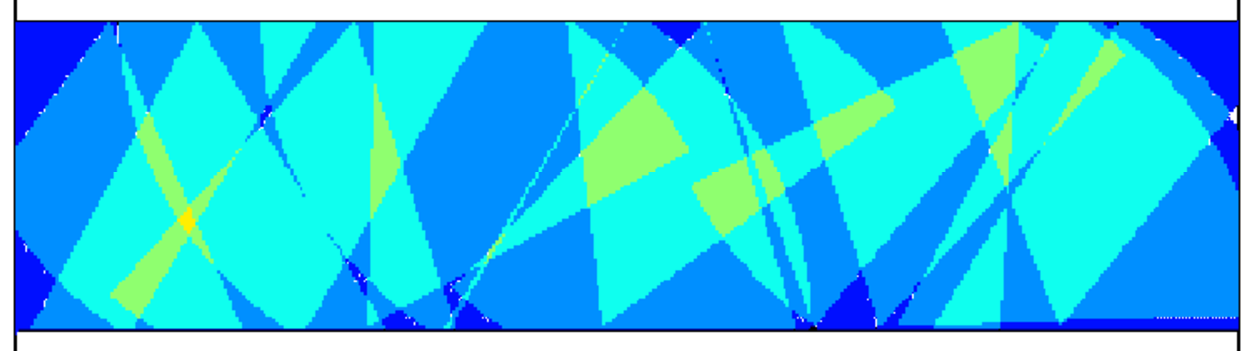}
\par\end{centering}

}
\par\end{centering}

\caption{\label{fig:can_fail} Position optimization with a possible failure. The
resulting sensor placements are obtained using the parameters $r=0.6$,
$h=0.005$, $\theta=\pi/3$. $N=50$ iterations of our algorithm
are used. Each sensor has a 50\% chance of not working, $p=0.5.$ 
Compared to Figure \ref{fig:no_fail}, the result shows more uniformly distributed coverages.
}
\end{figure}

\subsection{General examples for $p \neq 0$}\label{subsec:general_examples}
 
We consider  more general shapes of environment. The first is 
a narrow alley with a corner.  Given a sufficient number of sensors with the same failure rate, we will optimally position 
them on the walls to maximize an expectation of the collective coverage. The second is an arbitrary polygon,
and we consider the problem of optimal positioning on the boundary.

Figure \ref{fig:corner} illustrates how multiple sensors are optimally positioned to monitor the alley with a corner. 
4 sensors are randomly located on each wall facing the alley, i.e., a total of 16 sensors, as in subfigure (a). 
Red crosses indicate
the sensor locations, and sectors with dotted lines
describe the sensors\rq{} coverage.
The optimal positions found by our algorithm with 50 iterations are presented in subfigure (b). 
The coverage of sensors is uniformly distributed across the target area.
This environment will be generalized
to 3D in Section \ref{subsec:3d}.

\begin{figure}[H]
\begin{centering}
\hspace*{-1.0cm}\includegraphics[width=4cm]{figs/colorbar}
\par\end{centering}

\begin{centering}
\subfloat[Initial expected coverage area is 0.3313.]{\begin{centering}
\includegraphics[width=4cm]{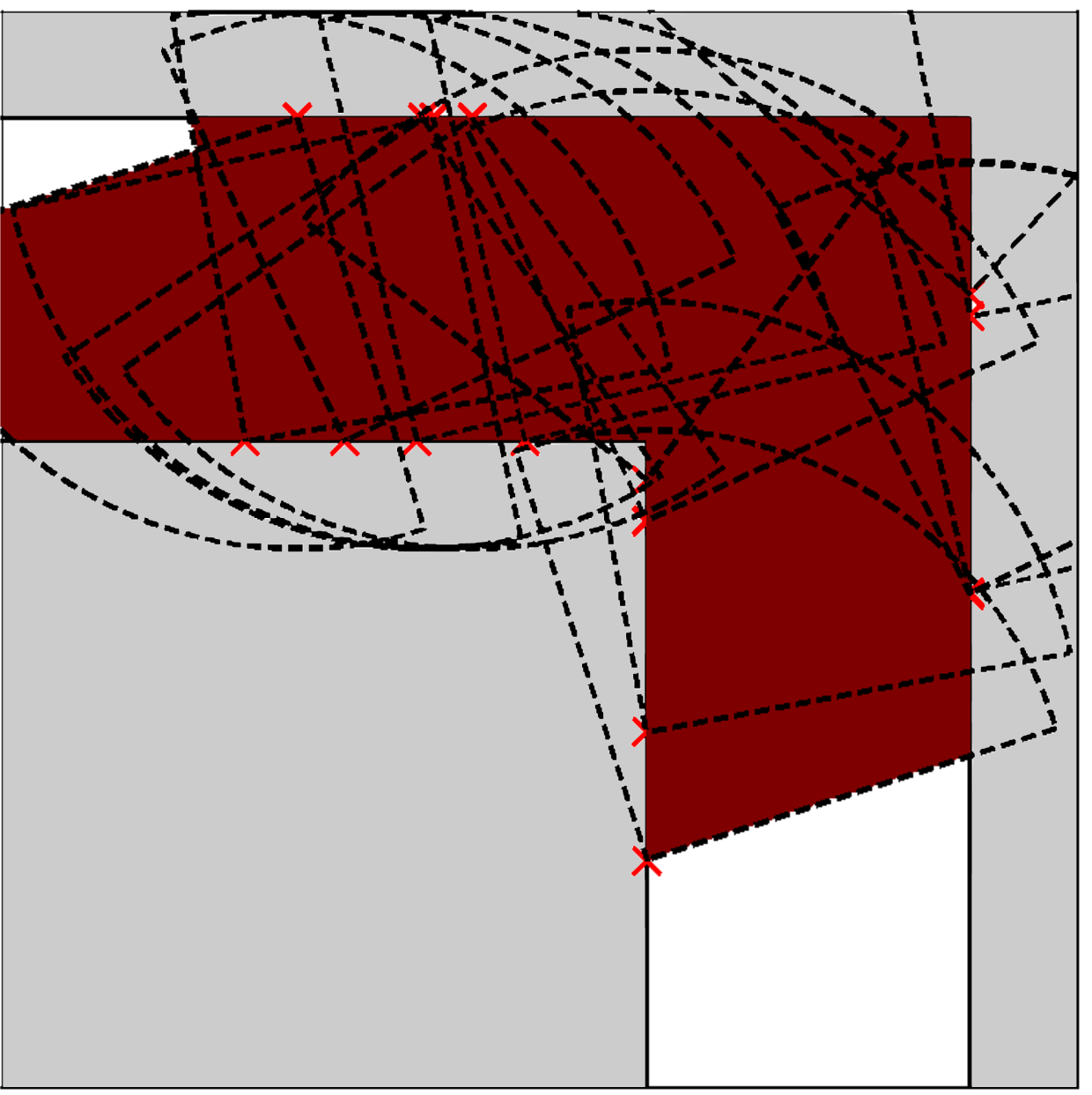}\;
\includegraphics[width=4cm]{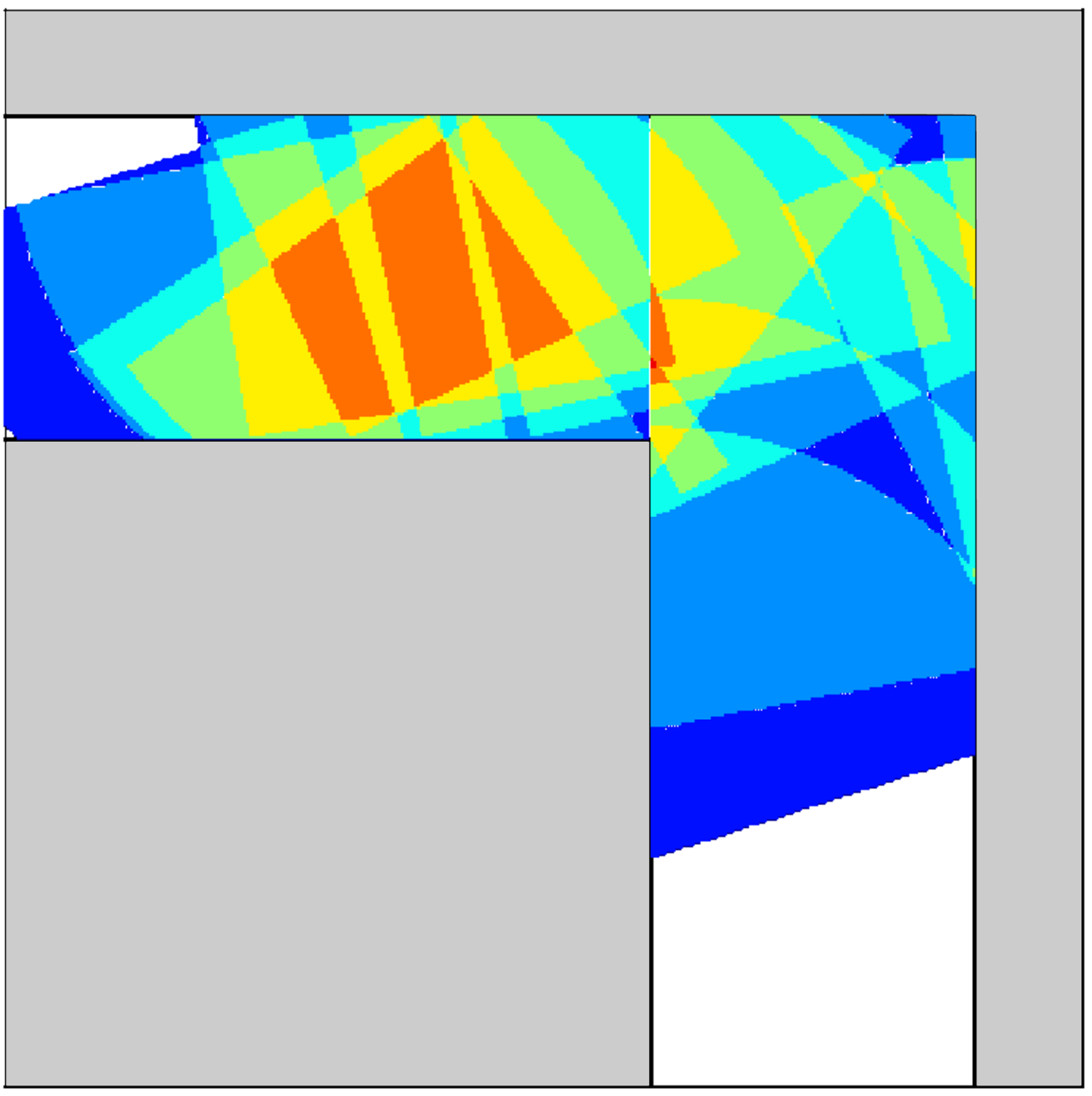}
\par\end{centering}

}\subfloat[After 50 iterations, resulting expected visible area is 0.9020.]{\begin{centering}
\includegraphics[width=4cm]{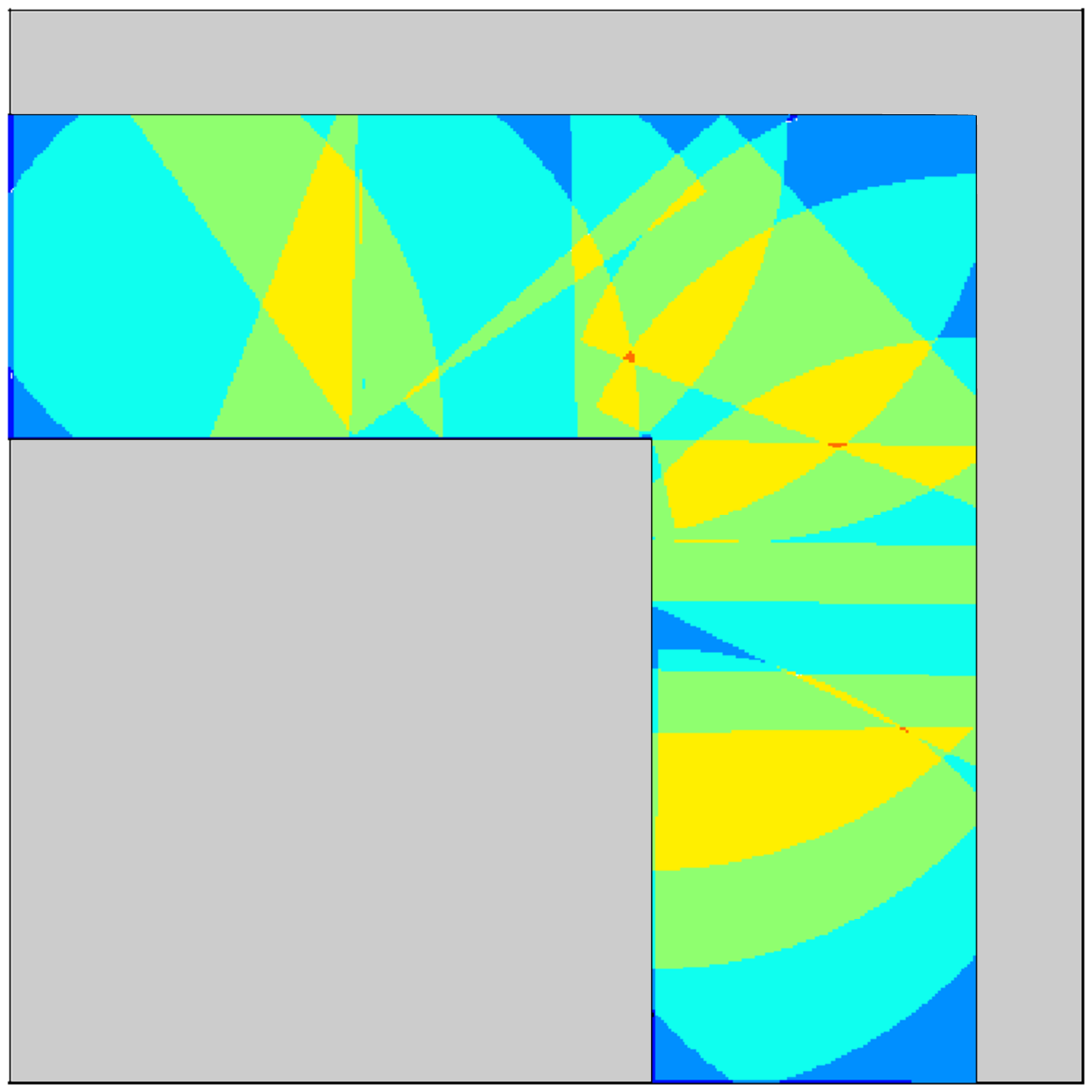}
\par\end{centering}

}
\par\end{centering}

\caption{\label{fig:corner} Optimal positioning in an alley with a corner. The resulting
sensor positions are obtained using the parameters $r=0.8$, $h=0.005$,
$\theta=\pi/2$, $p=0.5$. 16 sensors are considered. The result shows uniformly distributed coverages.}
\end{figure}

In Figure \ref{fig:complicated}, optimal sensor positioning in the environment with polygonal obstacles is considered. 
This scenario can be interpreted as positioning security surveillance cameras on the outside (e.g., rooftop) of urban buildings.
The polygons are shaded gray, and the failure rate of the sensors is taken as 0.5.
We apply our algorithm to optimally position 15 sensors on the boundary. The result in subfigure (b) is obtained with 200 iterations.
As shown in the figure, our algorithm can achieve a maximum area of coverage while handling complicated shapes of obstacles.

\begin{figure}[h]
\begin{centering}
\hspace*{-1.0cm}\includegraphics[width=4cm]{figs/colorbar}
\par\end{centering}

\begin{centering}
\subfloat[Initial expected coverage area is 0.2675.]{\begin{centering}
\includegraphics[width=4cm]{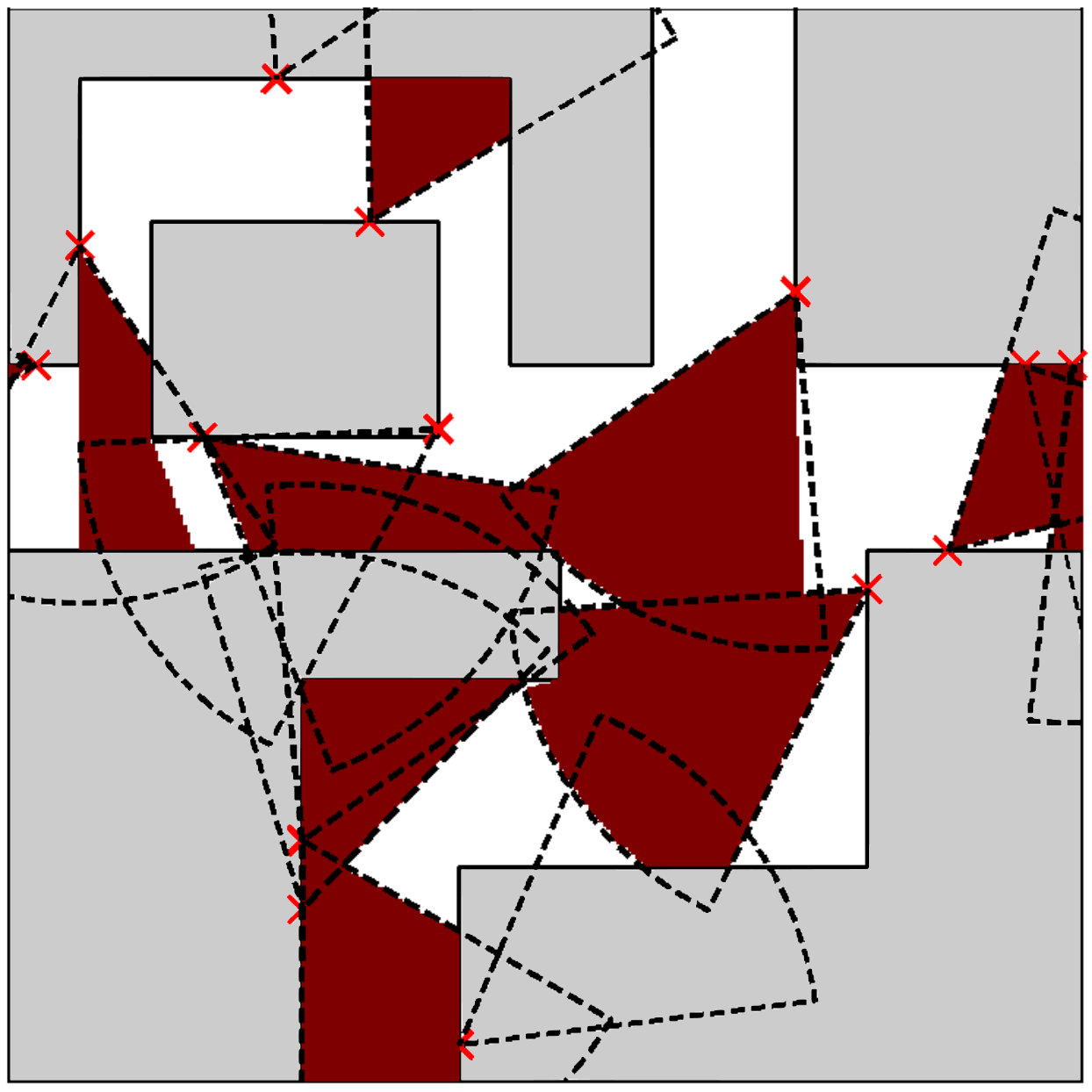}\;
\includegraphics[width=4cm]{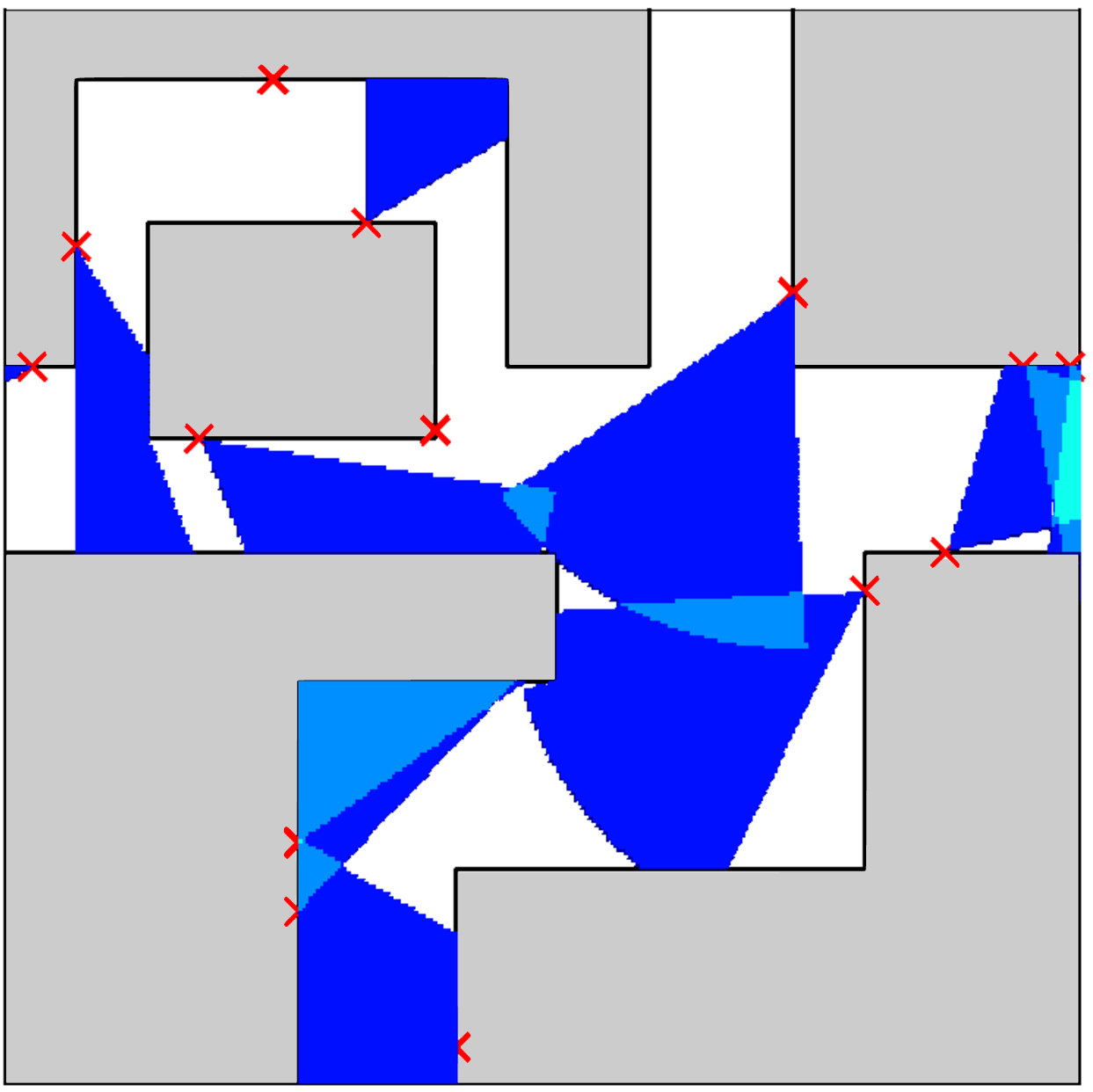}
\par\end{centering}

}\subfloat[After 200 iterations, resulting expected visible area is 0.5445.]{\begin{centering}
\includegraphics[width=4cm]{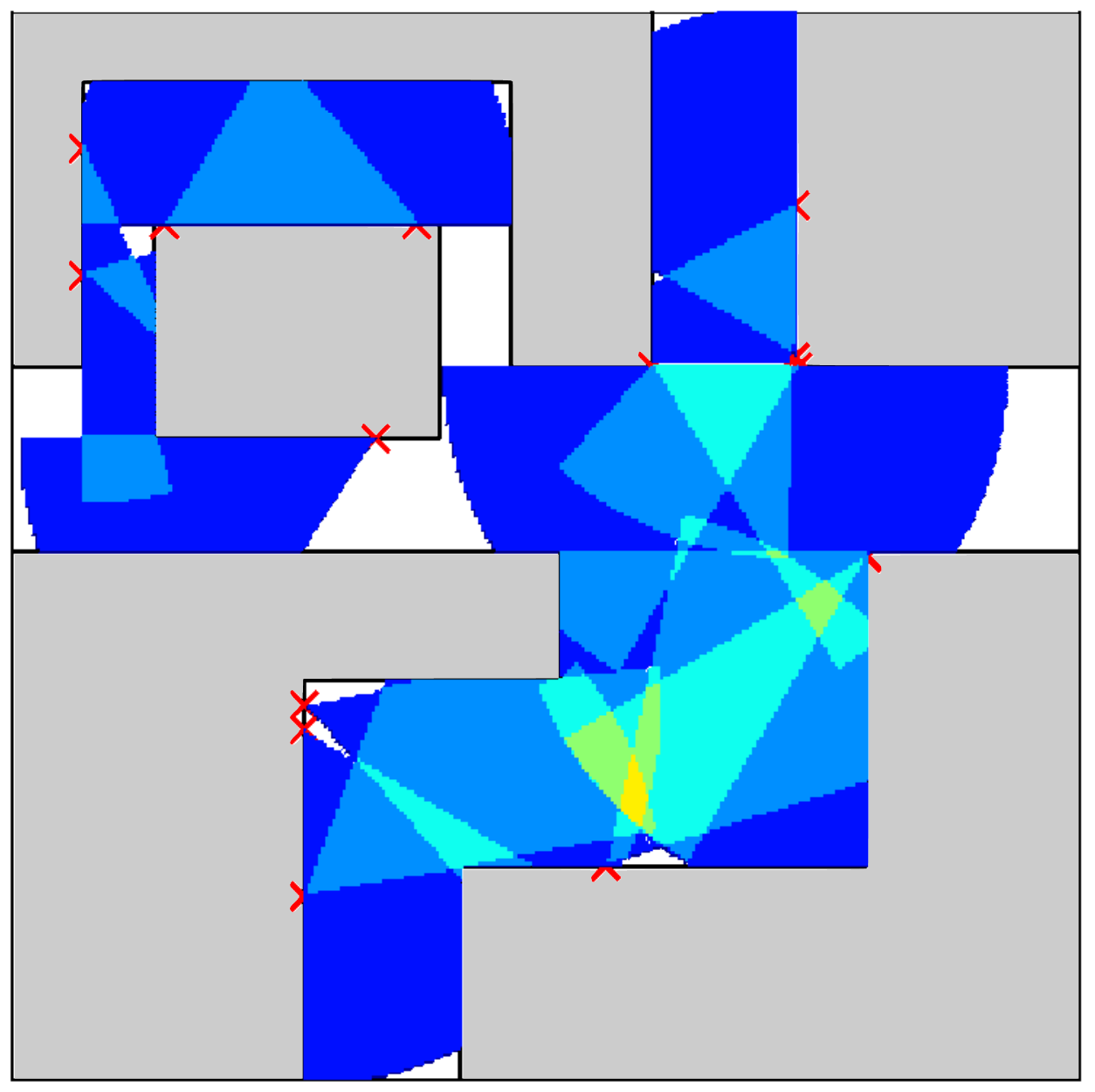}
\par\end{centering}

}
\par\end{centering}

\caption{\label{fig:complicated} Optimal positioning in an environment with polygonal obstacles. The resulting
sensor placements are obtained using the parameters $r=0.5$, $h=0.005$,
$\theta=\pi/3$, $p=0.5$. 15 sensors are considered.
This algorithm handles a complicated environment when sensors can move along the boundary.}
\end{figure}

\subsection{Access monitoring}
Access monitoring system plays an important role in securing a building 
and its surroundings. Well-designed monitoring systems 
are able to detect all of the suspicious activities while considering a possible failure of the sensors.
In this section, we consider a problem of optimally positioning the sensors to monitor the population entering a building.

We position the sensors on the boundary of the important building. The building is considered as the obstacle, and the sensors move along the boundary of the building to find the optimal positions. 
To use our approach with a possible failure rate, we construct the functional as
\begin{equation}
\mbox{\ensuremath{\mathbb{E}}}(V(x_{1},\cdots,x_{m},v_{1},\cdots,v_{m}))=\int_{D} \chi_A (y)H(\phi(y;x_{1},\cdots,x_{m},r_{1},\cdots,r_{m},v_{1},\cdots,v_{m}))(1-p(y))dy
\nonumber
%\label{eq:important_area}
\end{equation}
where $\chi_A$ is the indicator function of a set $A$, and $0<p(y)<1$ is a probability that a set of sensors fails to detect $y$.
Here, the set $A$ is a closed region which encloses the building so that the floating population must traverse this region in order to either enter or exit the building. Various shapes of $A$ can be considered, and  we use a narrow band around the building for simplicity. 
The main difference from the previous examples is that the area of full coverage only needs to be considered on a narrow band around the building as long as the access is fully covered.

In the following experiments, we consider a pentagon shaped building located at the center of the computational domain. 
Two red pentagons enclosing the building are the monitored region $A$. Note that one cannot move out or in to the building without passing the strip $A$.
We employ the algorithm which maximizes the above functional to find 
the optimal position as well as the viewing direction of sensors while considering the sensor's possible failure.

In the case of symmetric objects, one can consider solving the problem more efficiently
by adding a symmetry constraint to the feasible position.
Figures~\ref{fig:corner-tube-sym}(a-b) show such an example where we position three sensors to each
side symmetrically. One sensor is fixed at the center with a viewing direction normal to the boundary, 
and the other two move with symmetry: (1) when one sensor moves closer to the corner, the other moves to the
other corner with the same distance; (2)  when the viewing direction of one sensor rotates clockwise, the one of the other rotates 
counterclockwise at the same angle. 
This problem requires a 2-dimensional optimization (location and viewing direction) to position two sensors away from the center.  This case,
from the initial condition in Figure~\ref{fig:corner-tube-sym}(a), the optimal solution is computed as Figure~\ref{fig:corner-tube-sym}(b) with the expected coverage of 0.3366.  Under this symmetry constraint, we find that the optimal location of three sensors divides the edge in the ratio 0.29:0.5:0.71. Interestingly, the optimal result does not equally divide the edge into four as one may expect.

\begin{figure}
\begin{centering}
\hspace*{-1.0cm}\includegraphics[width=4cm]{figs/colorbar}
\par\end{centering}

\begin{centering}
\subfloat[Initial expected coverage area is 0.2874.]{\begin{centering}
\includegraphics[width=4cm]{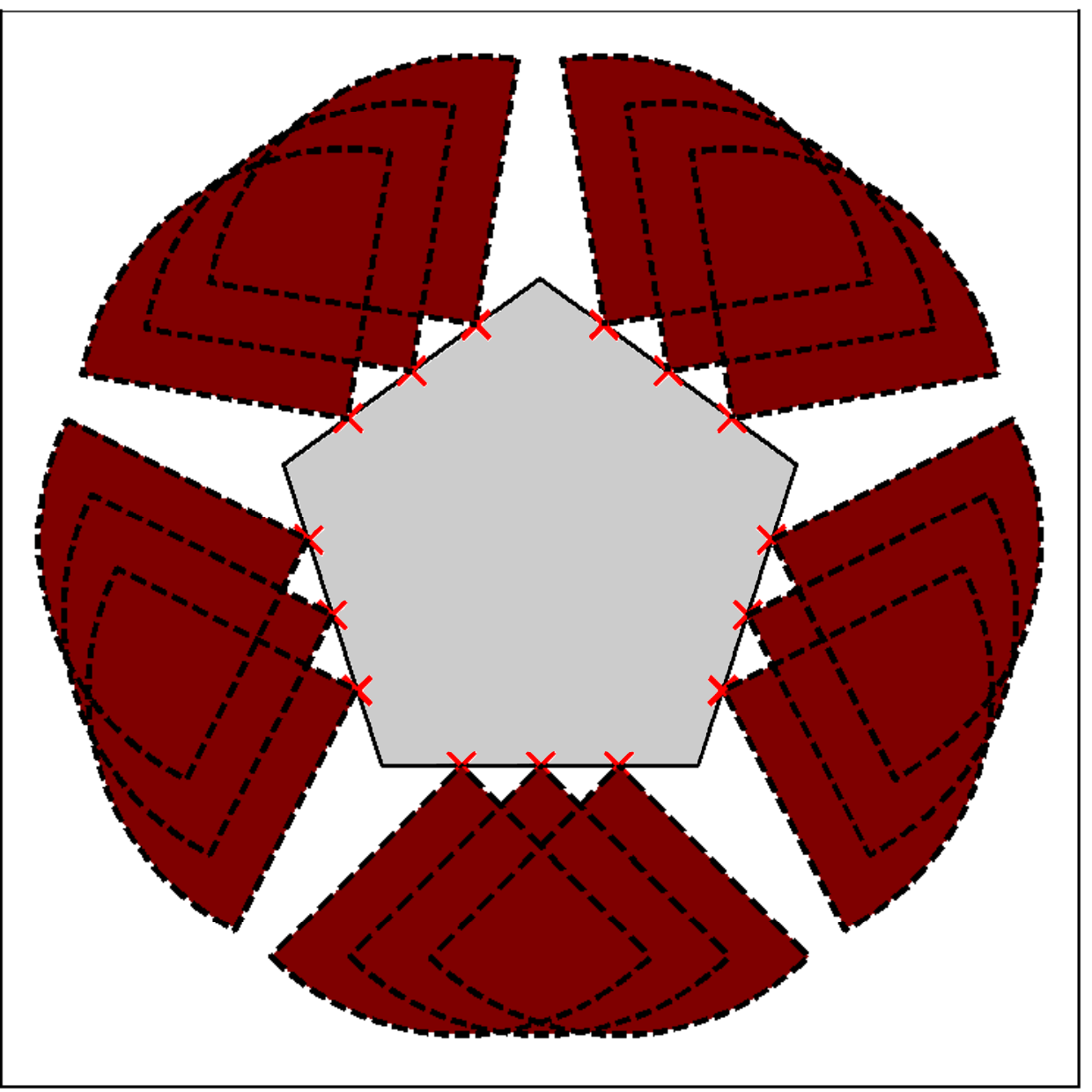}\includegraphics[width=4cm]{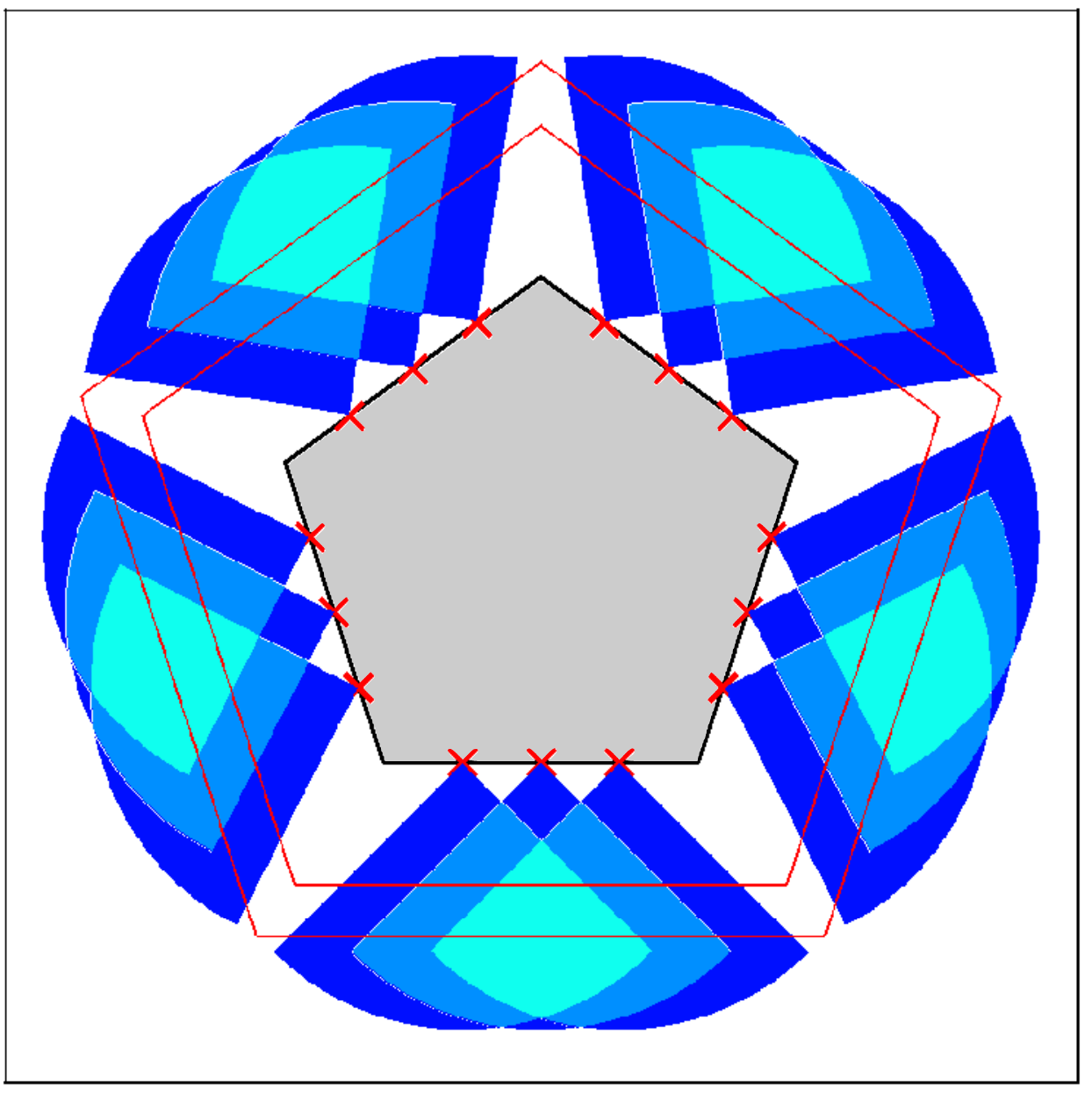}
\par\end{centering}

}\subfloat[After 100 iterations, 0.3366.]{\begin{centering}
\includegraphics[width=4cm]{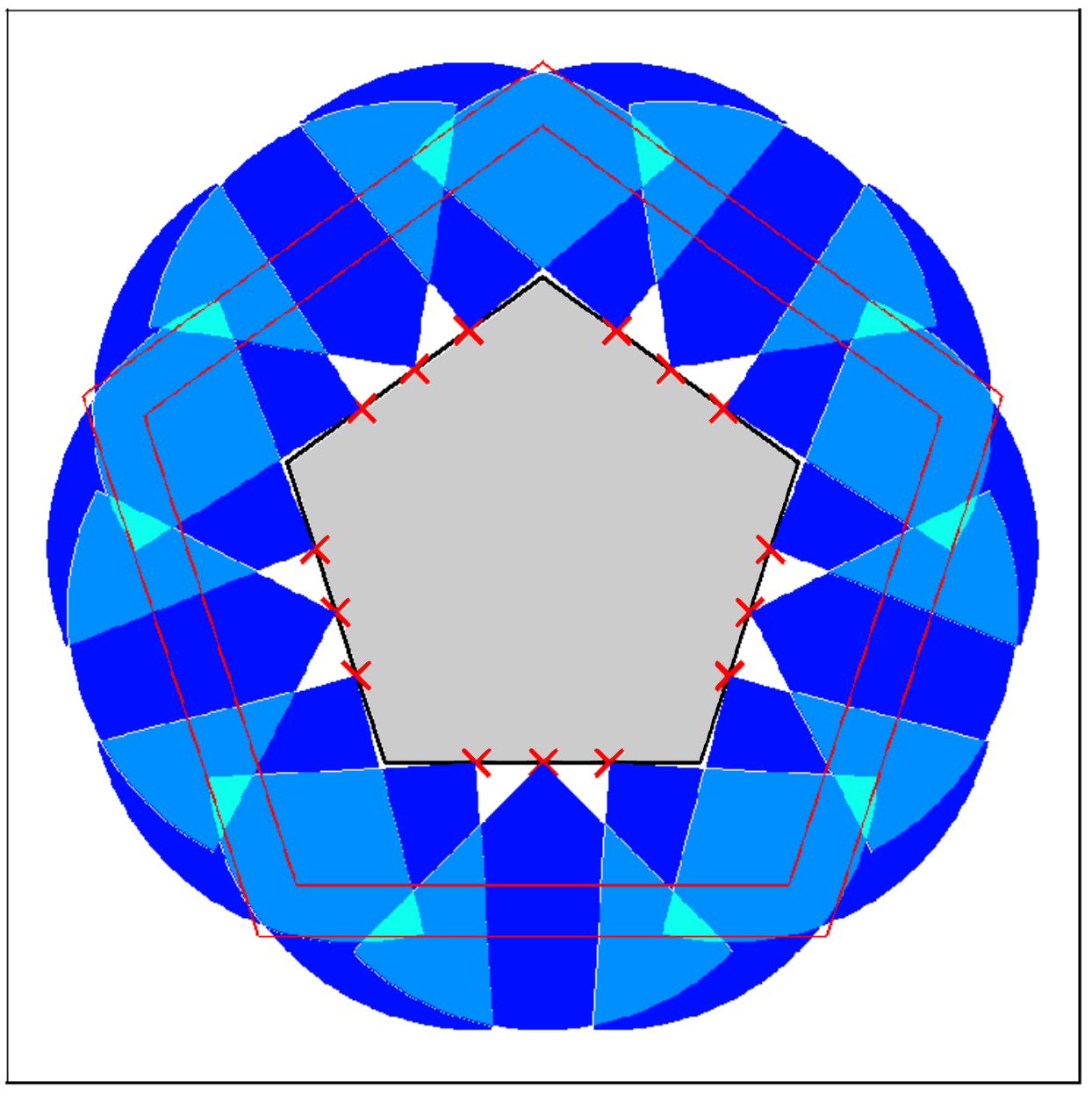}
\par\end{centering}

}
\par\end{centering}

\begin{centering}
\subfloat[Initial expected coverage area is 0.2422.]{\begin{centering}
\includegraphics[width=4cm]{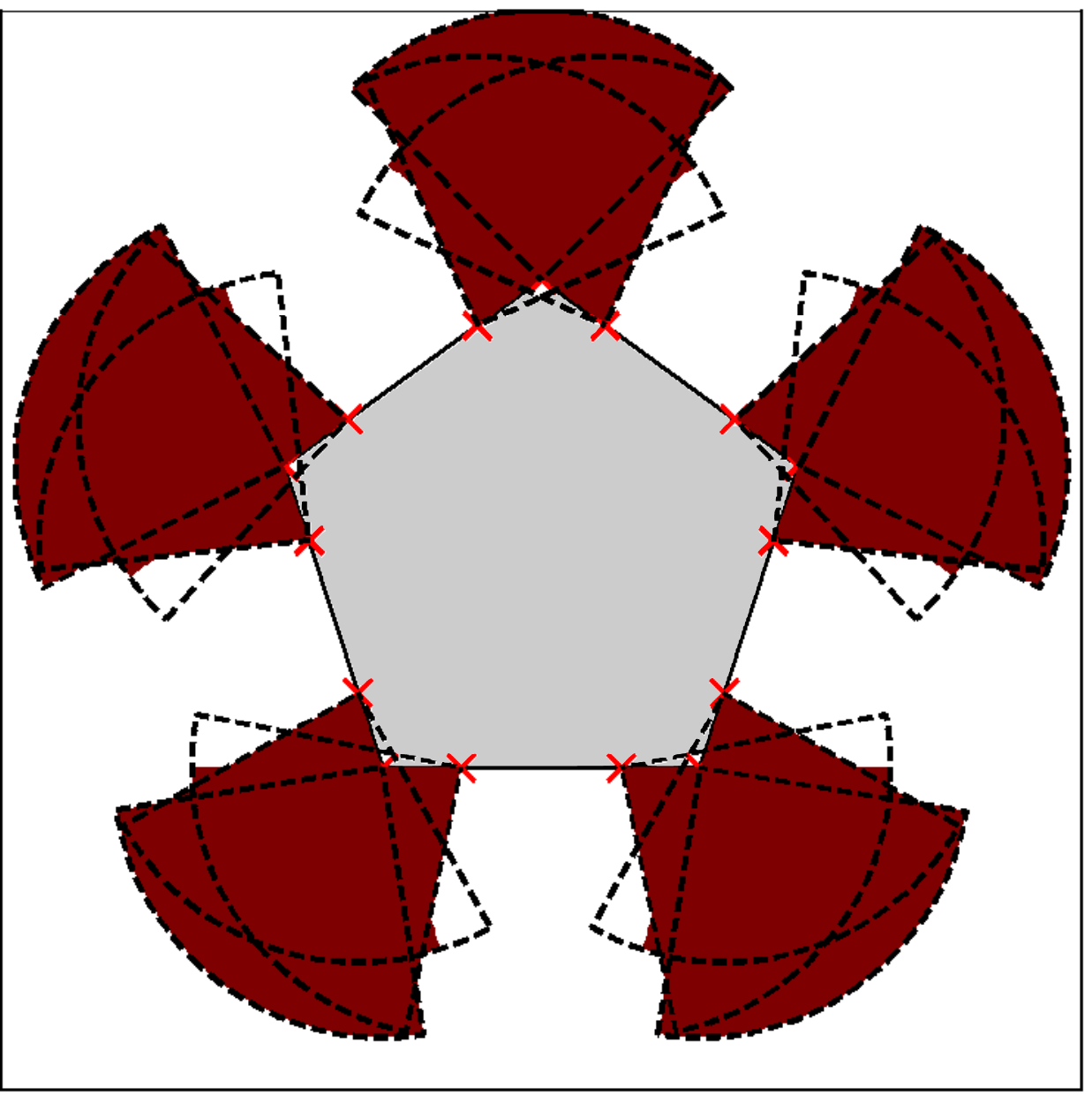}\includegraphics[width=4cm]{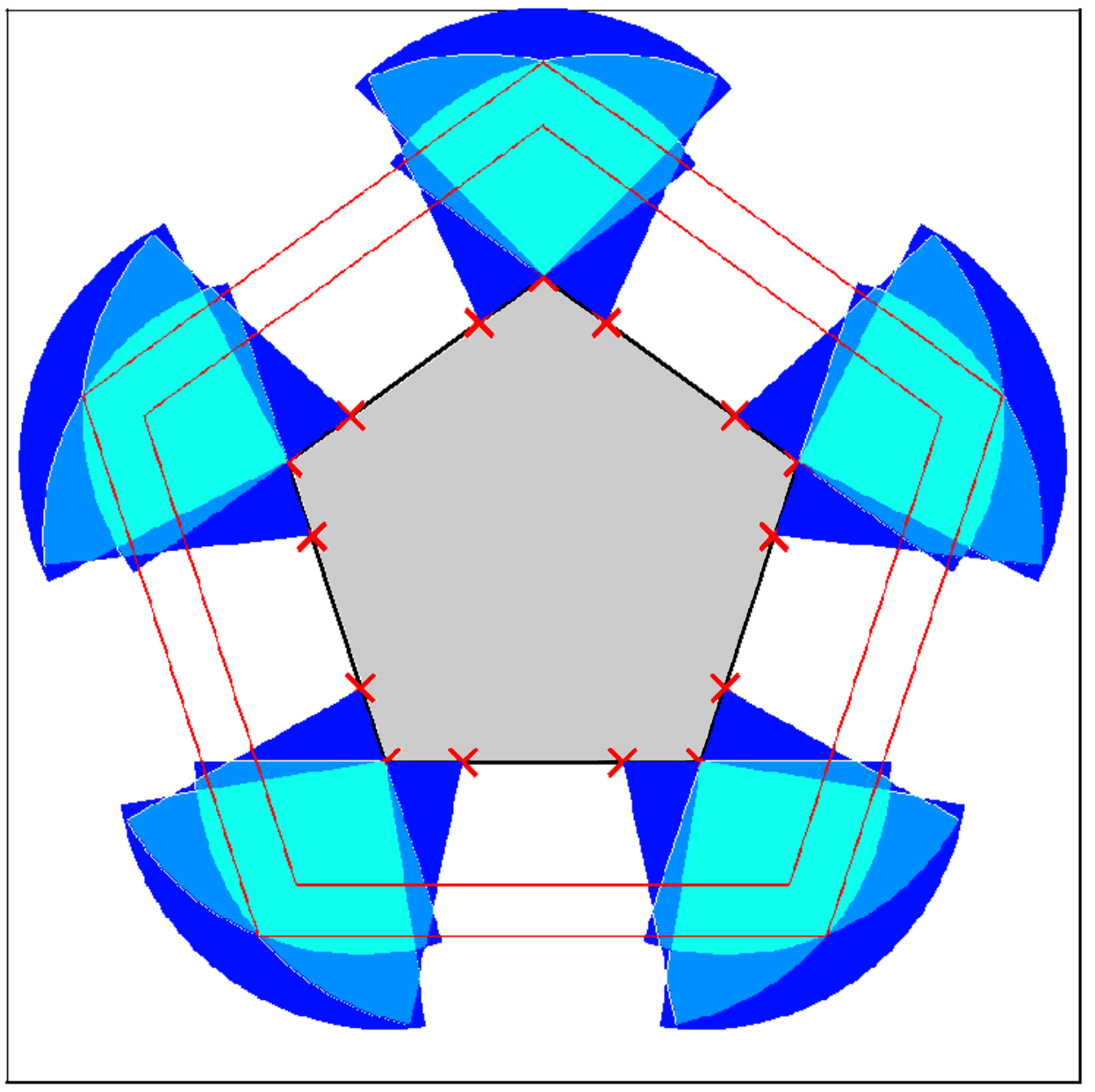}
\par\end{centering}

}\subfloat[After 100 iterations, 0.3180.]{\begin{centering}
\includegraphics[width=4cm]{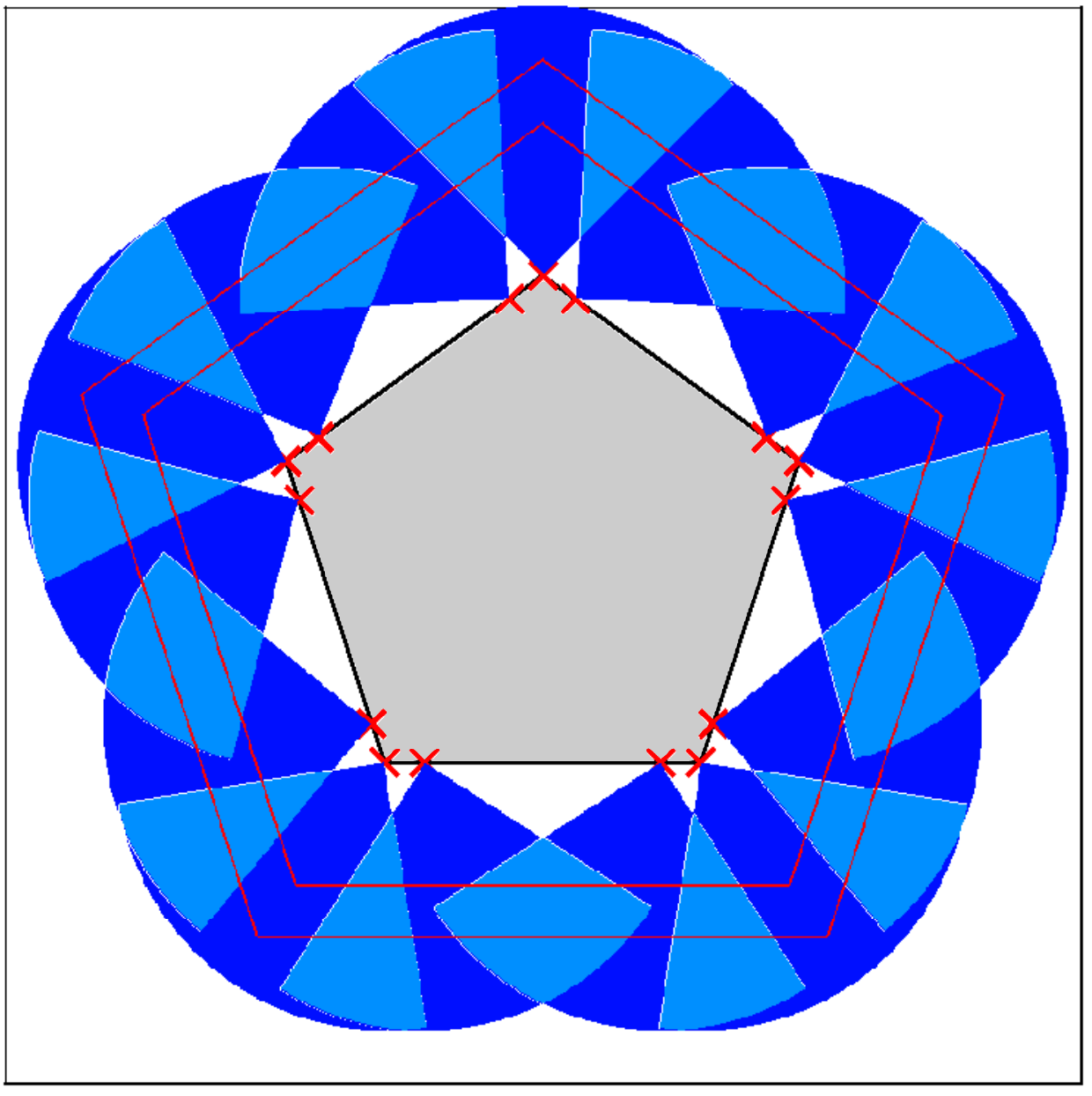}
\par\end{centering}

}
\par\end{centering}

\caption{\label{fig:corner-tube-sym} Entrance security monitoring for the pentagon under two symmetry constraints. The resulting
sensor positions are obtained with 15 sensors using the same parameters 
$r=0.5$, $h=0.0025$, $\theta=\pi/2$, $p=0.5$. 
%as Figure \ref{fig:corner-tube}.
There are 2 degrees of freedom in optimization. 
}
\end{figure}

As a remark, one can also consider different types of symmetry. Figures~\ref{fig:corner-tube-sym}(c-d) show another of such example. 
Each sensor is fixed at the corner with a viewing direction exactly opposite to the center. Other two sensors move with symmetric on the edge as before. 
The result of our algorithm is given in Figure~\ref{fig:corner-tube-sym}(d). The resulting expected value is not as high as the one in Figure~\ref{fig:corner-tube-sym}(b).
% or Figure A (c), each being 3.xxx and 3.xxx. 

We can solve the problem without any symmetry constraint. This is more suitable for the general shapes. 
In this case, however, the degree of freedom is increased from 2 to 30 (position and viewing direction of 15 sensors). 
With intermittent diffusion and gradient descent for all variables, 
from the initial condition in Figure~\ref{fig:corner-tube}(a), the optimal solution is computed as Figure~\ref{fig:corner-tube}(b) with the expected coverage of 0.3244. The result  distributes the sensors and adjusts the viewing directions to uniformly cover the entire area of $A$.
This result provides the better expected coverage area than the one of Figure~\ref{fig:corner-tube-sym}(d) but the worse than the one of Figure~\ref{fig:corner-tube-sym}(b). 
Although the expected coverage area is not as sharp as the one with symmetry constraint, this result shows that our algorithm can 
handle various cases and approximate the globally optimal solution. 

\begin{figure}
%\begin{centering}
%\hspace*{-1.0cm}\includegraphics[width=4cm]{figs/colorbar}
%\par\end{centering}

\begin{centering}
\subfloat[Initial expected coverage area is 0.2881.]{\begin{centering}
\includegraphics[width=4cm]{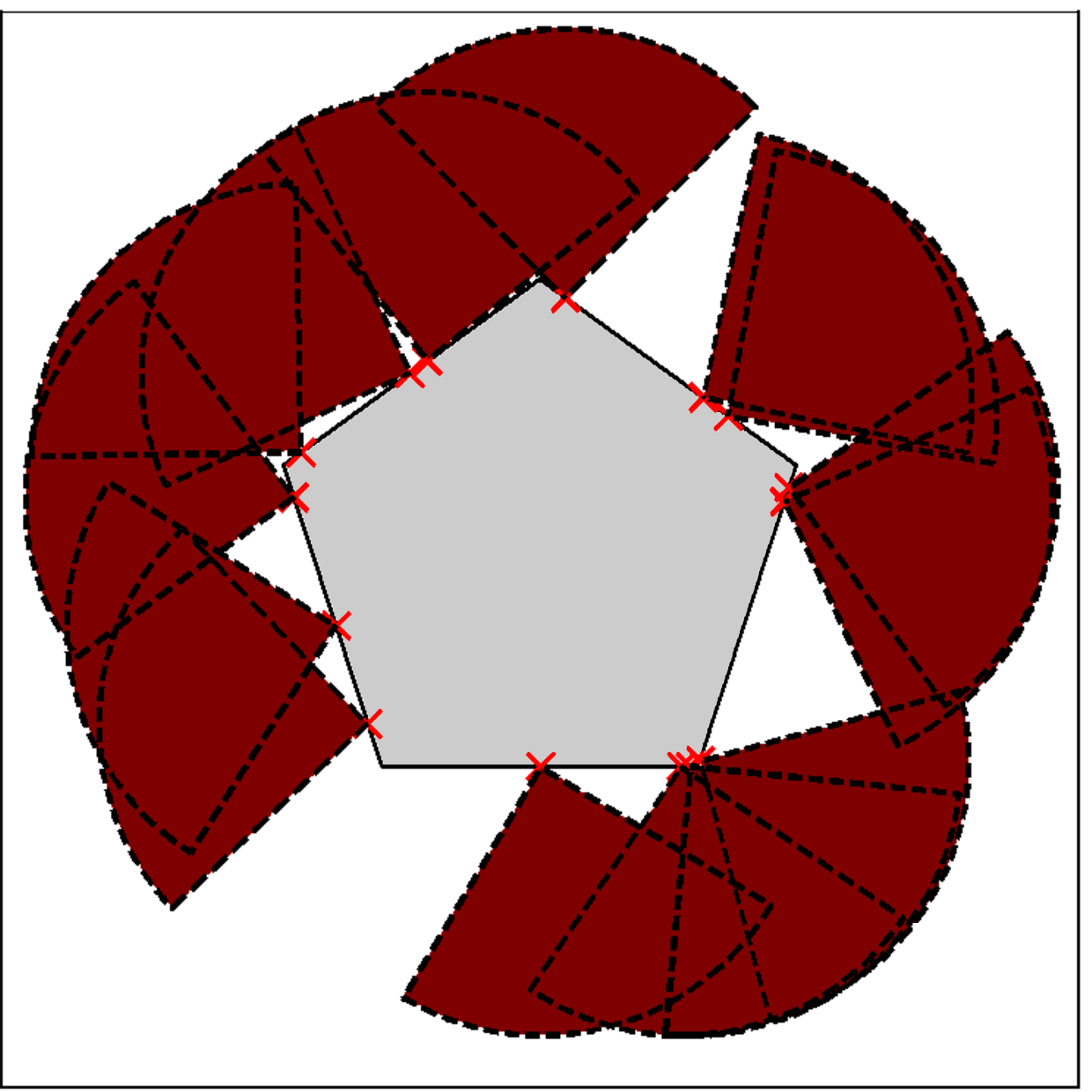}\includegraphics[width=4cm]{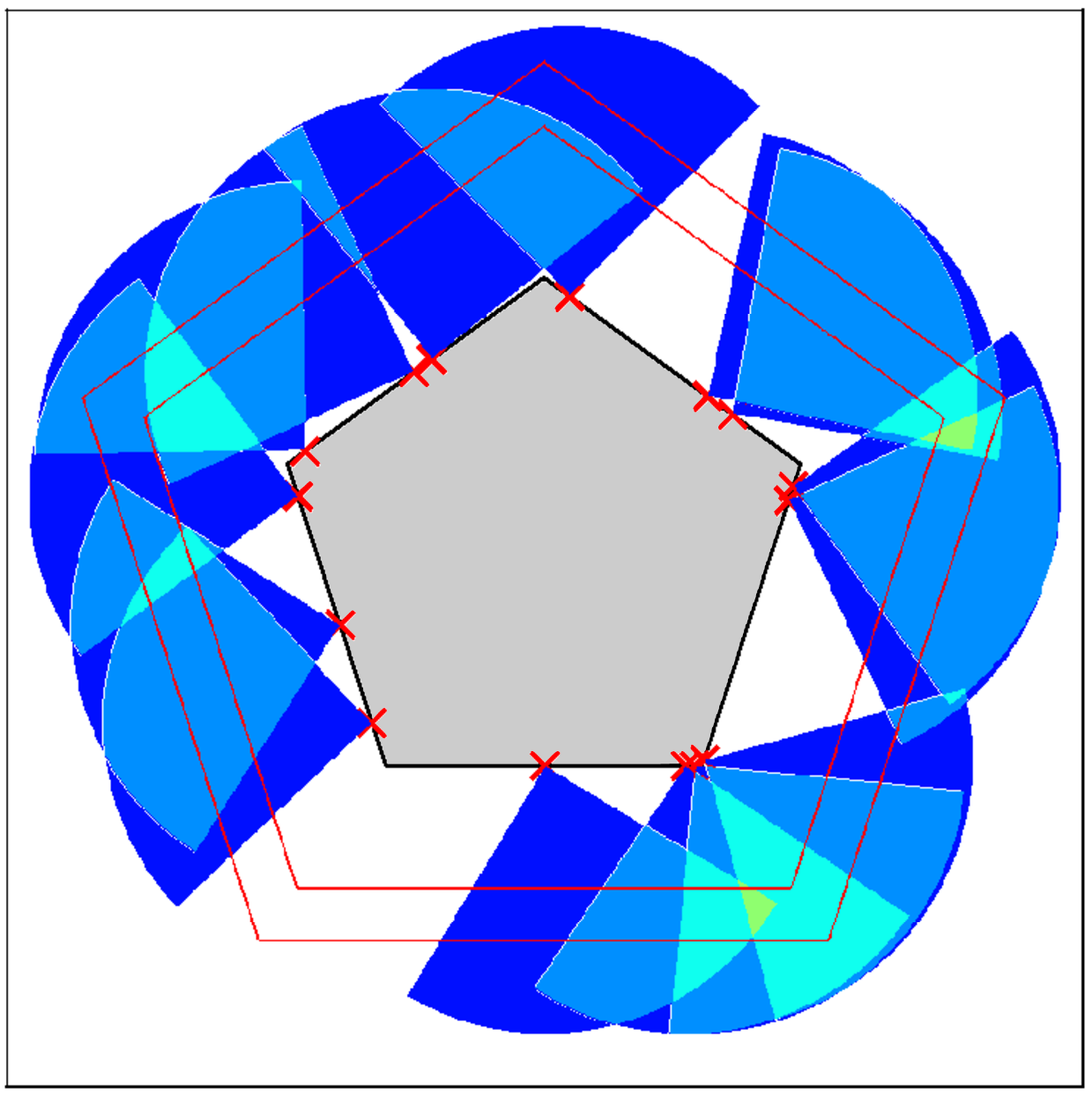}
\par\end{centering}

}\subfloat[After 100 iterations, 0.3244.]{\begin{centering}
\includegraphics[width=4cm]{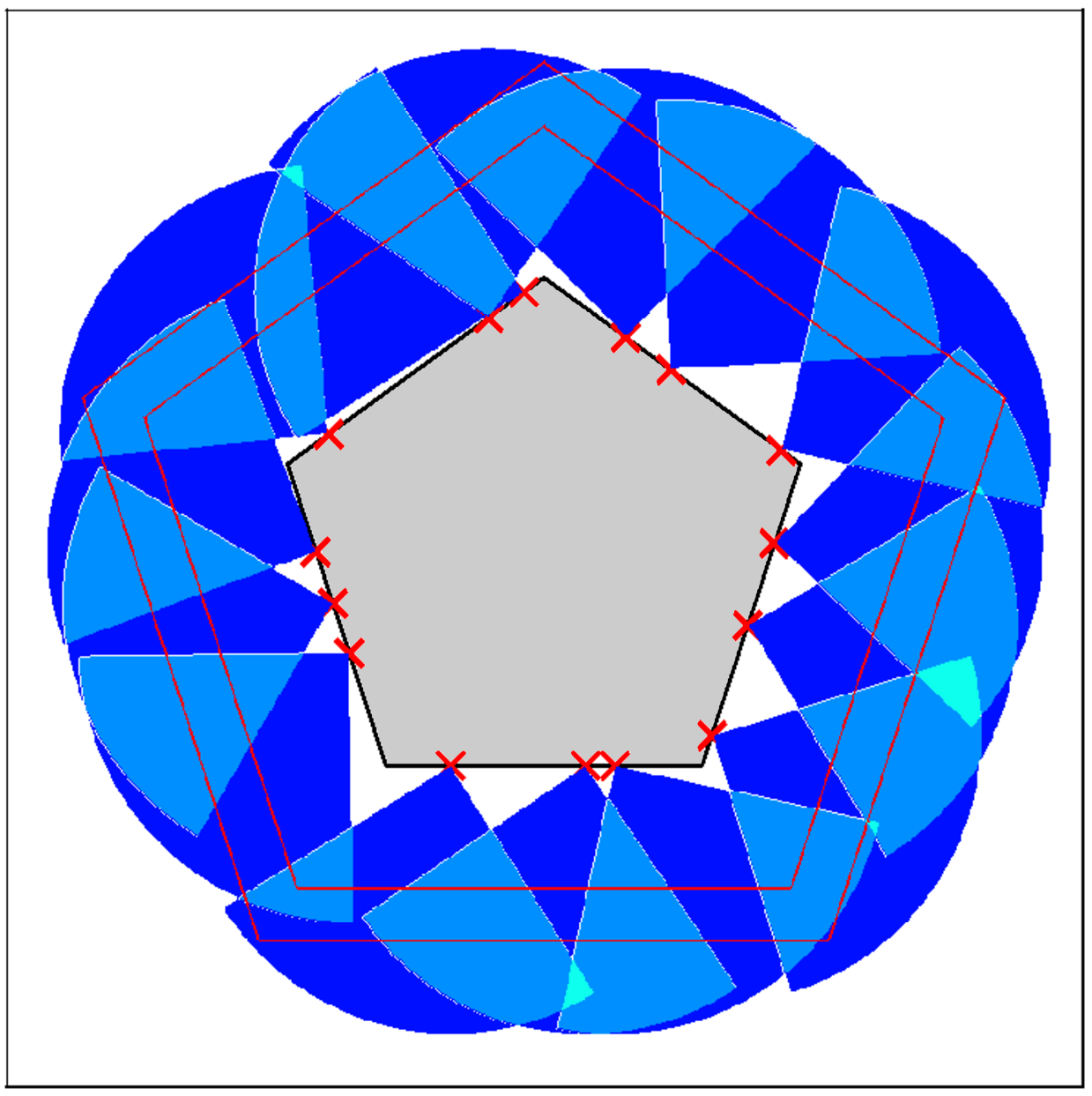}
\par\end{centering}

}
\par\end{centering}

\caption{\label{fig:corner-tube} Entrance security monitoring for the pentagon. The resulting
sensor positions are obtained with 15 sensors using the same parameters as Figure~\ref{fig:corner-tube-sym}.
%$r=0.5$, $h=0.0025$, $\theta=\pi/2$, $p=0.5$. 
There are 30 degrees of freedom in optimization. 
}
\end{figure}

\subsection{Important regions}
We consider a particular problem which some regions in the monitored area are more important than others. 
In other words, the various weights of importance are assigned and given a priori. To apply the previous algorithm, we modify the objective functional as
\begin{equation}
\mbox{\ensuremath{\mathbb{E}}}(V(x_{1},\cdots,x_{m},v_{1},\cdots,v_{m}))=\int_{D}w(y)H(\phi(y;x_{1},\cdots,x_{m},r_{1},\cdots,r_{m},v_{1},\cdots,v_{m}))(1-p(y))dy
\nonumber
%\label{eq:important_area}
\end{equation}
where the weight $w(\cdot)$ is distributed over the environment. 

This can take care of the scenario where there are blind spots in the environment we want to monitor.
This setup resembles a realistic case, e.g., sensor positioning for grocery market and book store. To set-up this environment, we first placed a stationary long-range sensor (which can represent the location of the casher), and computed the blind spots using the visibility computation using Algorithm \ref{PDE-based-algorithm}.   In Figure \ref{fig:shop0.5} (a), the circle represents the cashier's location, and the red lines behind the gray rectangles (e.g. book shelves) represents the edge of visibility boundary.  In addition, near the casher is an important area to monitor, this is represented by the red square box enclosing the red circle.

%We suppose that a stationary long-range sensor is located in advance and that there are obstacles which block this sensor\rq{}s coverage. In Figure \ref{fig:shop0.5}, a long-range sensor is located at the circle, and the occluded region from this sensor is obtained by Algorithm \ref{PDE-based-algorithm}. We then take two important regions: one is the square area containing a long-range sensor, and the other is the occluded area due to the rectangular obstacles. 

The objective is to optimally position the sensors to monitor the important regions, i.e., the two separate areas, each enclosed by the red lines.  This mission can be easily tackled using our approach. 
We assign a weight $1$ only on the important regions and $0$ on the others, then apply Algorithm \ref{alg:Multiple-movable observer-algorithm} to maximize the above functional.
Considering the strategy of sensor's possible failure, the initial positions in the subfigure (b) are changed into the subfigure (c) with 50 iterations. The failure rate is chosen $p=0.5$.  As shown in the result, full coverage of the area is not achieved although the maximum of the expected coverage  is attained.  
%For the sensors which can cover non-covered regions, 
Instead of covering the small non-covered regions, sensors make a large overlap in the covered regions for maximum expectation because the contribution of the overlapping areas outweighs that of the non-covered areas.  This is also due to the difference between the shape and the topology of the important region and the coverage shape of the sensors. 

When the failure rate is lower, one can achieve full coverage. 
Figure \ref{fig:shop0.1} shows an example with a smaller failure rate $p=0.1$.  We apply Algorithm \ref{alg:Multiple-movable observer-algorithm} to maximize the coverage area, and compared to the case when $p=0.5$, the important regions are both fully covered by the sensors, and the result gives a higher expected coverage.

\begin{figure}
\begin{centering}
\hspace*{-0.6cm}\includegraphics[width=4cm]{figs/colorbar}
\par\end{centering}

\begin{centering}
\subfloat[Important regions.]{\begin{centering}
\includegraphics[width=4cm]{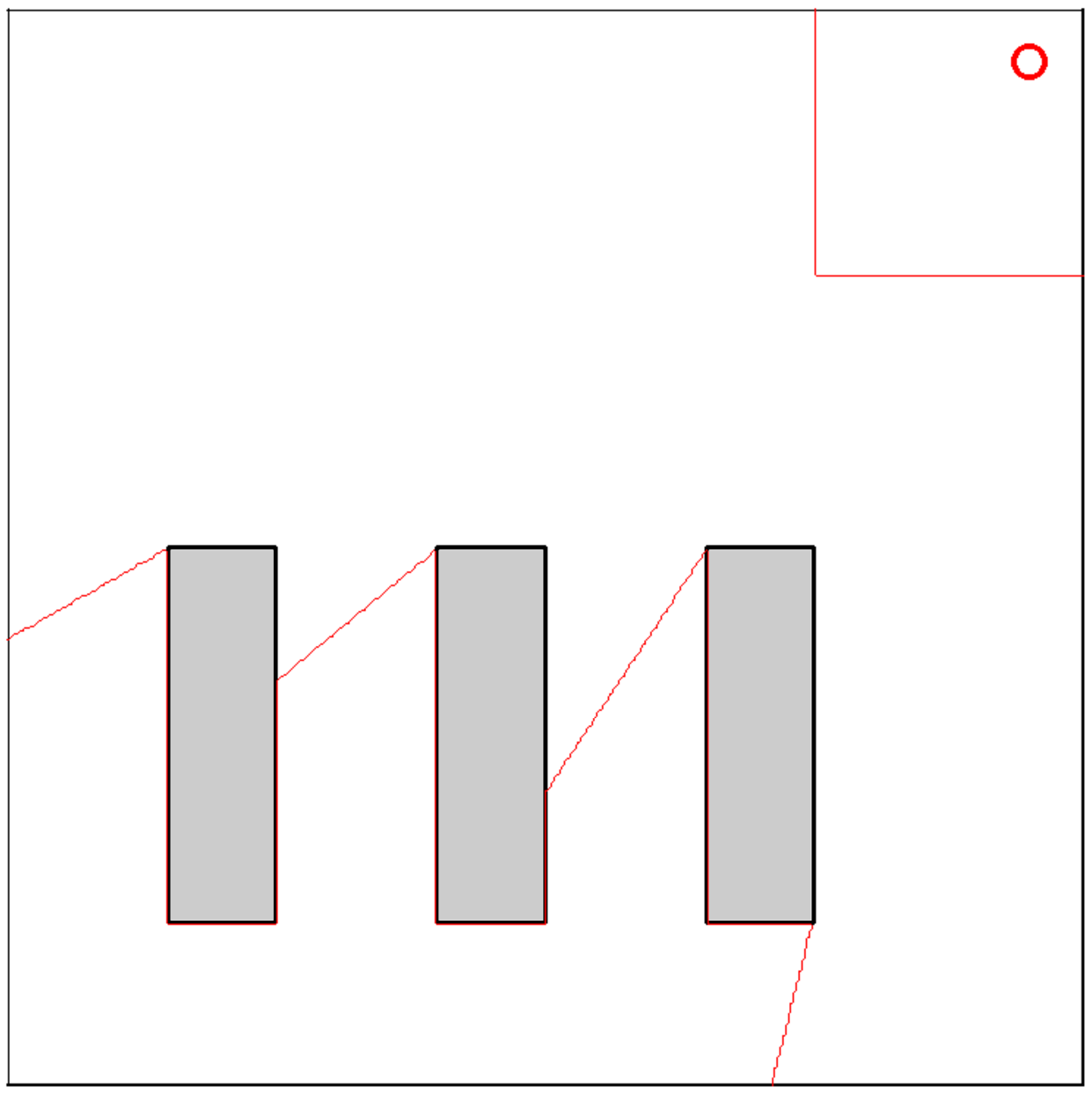}
\par\end{centering}
}
\subfloat[Initial expected coverage area on the red regions is 0.5596.]{\begin{centering}
\includegraphics[width=4cm]{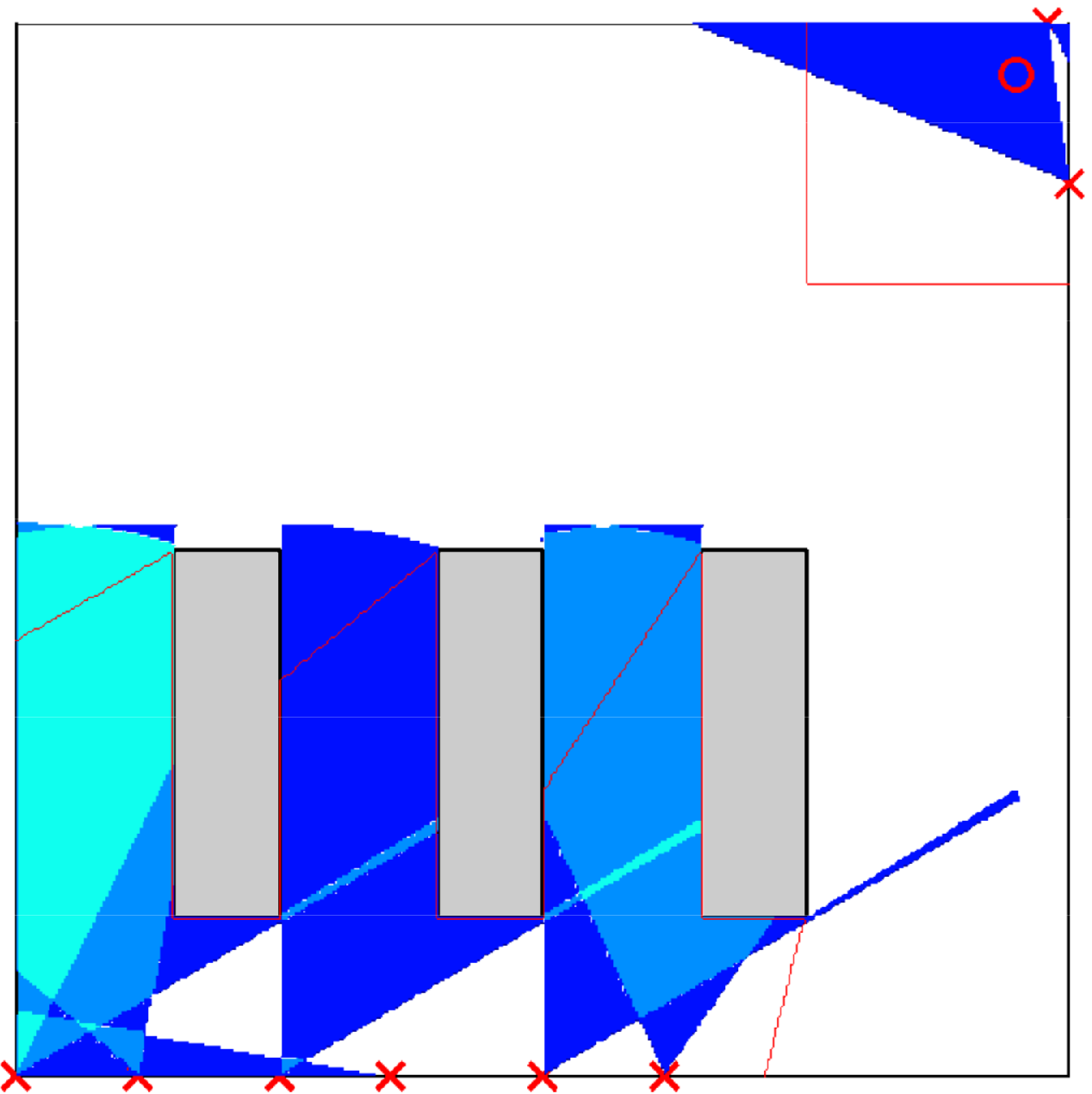}
\par\end{centering}
}
\subfloat[After 50 iterations, 0.8712.]{\begin{centering}
\includegraphics[width=4cm]{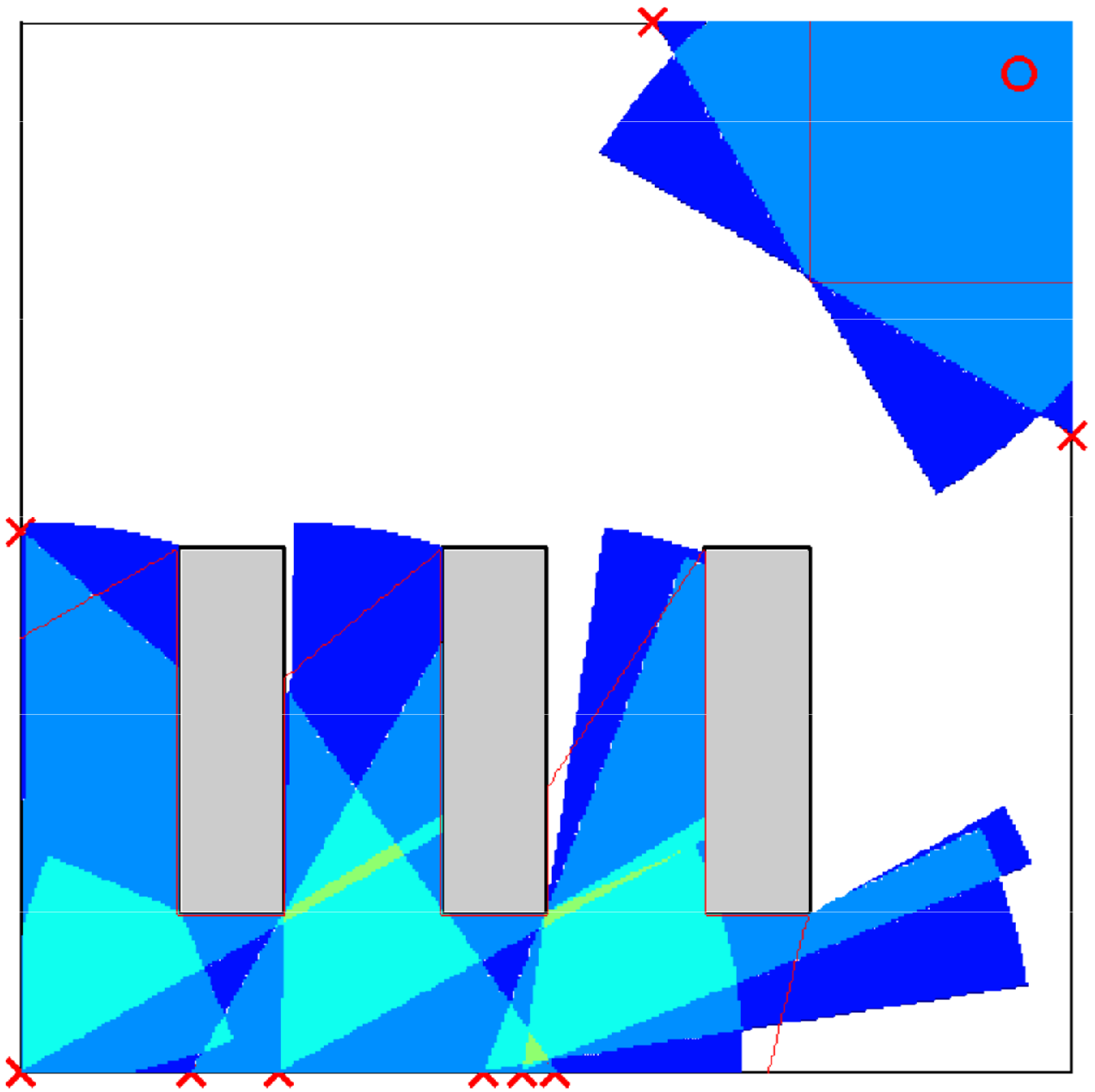}
\par\end{centering}

}
\par\end{centering}

\caption{\label{fig:shop0.5} Important region monitoring with $p=0.5$. The resulting sensor
placements are obtained with 9 sensors using the parameters $r=1.05$,
$h=0.005$, $\theta=\pi/3$. Most of the important regions are covered, but there exist non-covered regions.}
\end{figure}

\begin{figure}
%\begin{centering}
%\hspace*{-0.7cm}\includegraphics[width=4cm]{figs/colorbar}
%\par\end{centering}

\begin{centering}
\subfloat[Initial expected coverage area on the red regions is 0.8985.]{\begin{centering}
\includegraphics[width=4cm]{figs/9ID_store_many_p0\lyxdot5_105_0}
\par\end{centering}
}
\subfloat[After 50 iterations, 1.1298.]{\begin{centering}
\includegraphics[width=4cm]{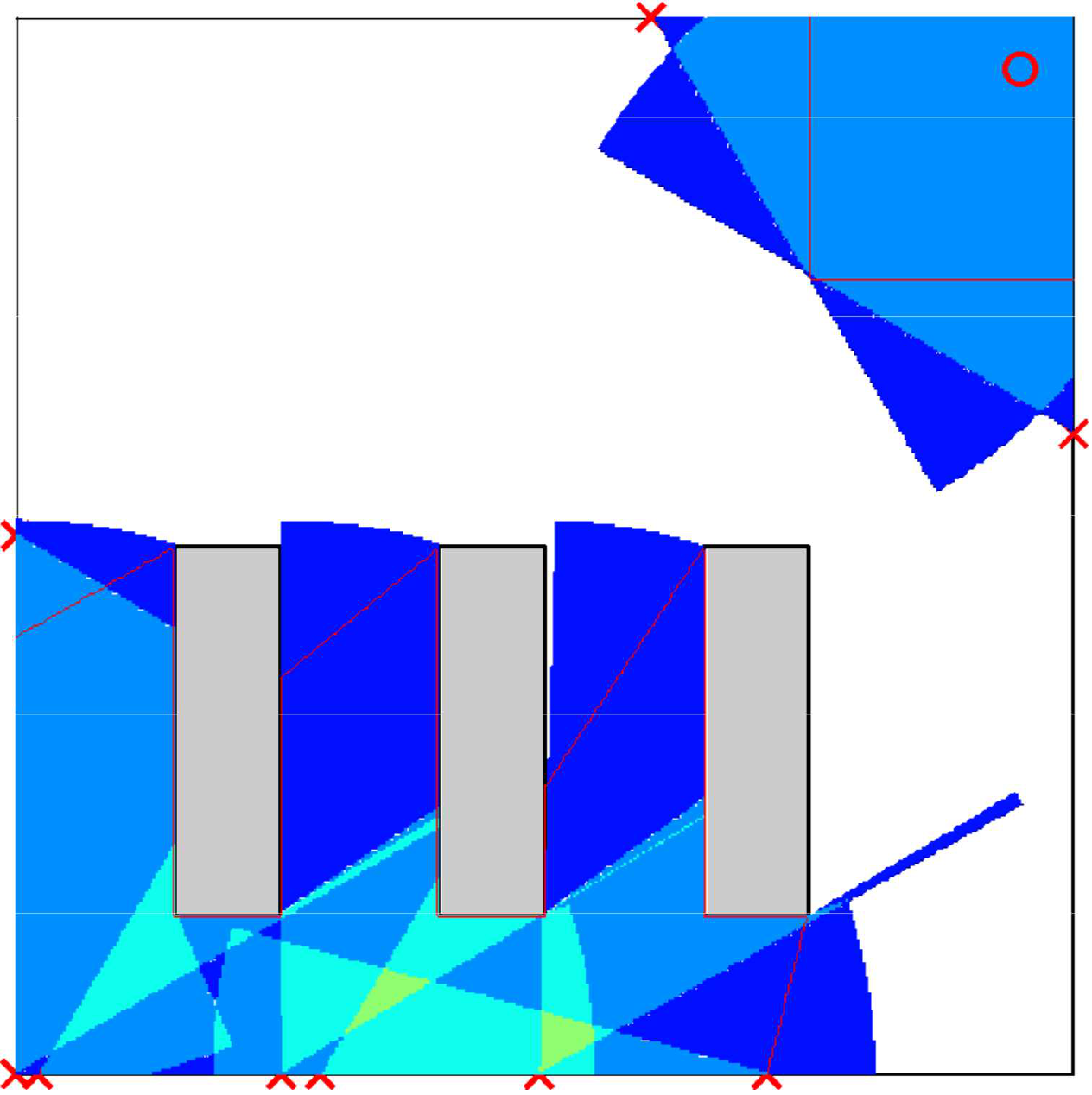}
\par\end{centering}

}
\par\end{centering}

\caption{\label{fig:shop0.1} Important region monitoring with $p=0.1$. The resulting sensor
placements are obtained with 9 sensors using the same parameters as Figure~\ref{fig:shop0.5}.
The important regions are fully covered.}
\end{figure}

\subsection{Higher dimensional example\label{subsec:3d}}
 
The examples presented so far were restricted to 2-dimensional environments.  The proposed method can be easily generalized to the higher dimensions. By modifying the level set formulation in Section~\ref{sec:Coverage-optimization-GD}, the previous approach can be applicable to the higher dimensional setup. 
For simplicity, we focus on 3-dimensional which addresses a real world environment.
The higher ($>3$) dimensions can be considered in a similar manner.

%With an effort to 3D, we consider the design of the surveillance system
%in the practical environment. 
A sensor in 3-dimensional space is given
in Figure \ref{fig:3D-sensor}. The coverage of a sensor is modeled
by a spherical sector whose apex point is located at the sensor location.
That is, as described in the right subfigure, we generalize the 2-dimensional
sensor in Figure \ref{fig:A-sensor} to the 3-dimension by rotating it about
the central axis.

\begin{figure}[H]
\begin{centering}
\includegraphics[width=10cm]{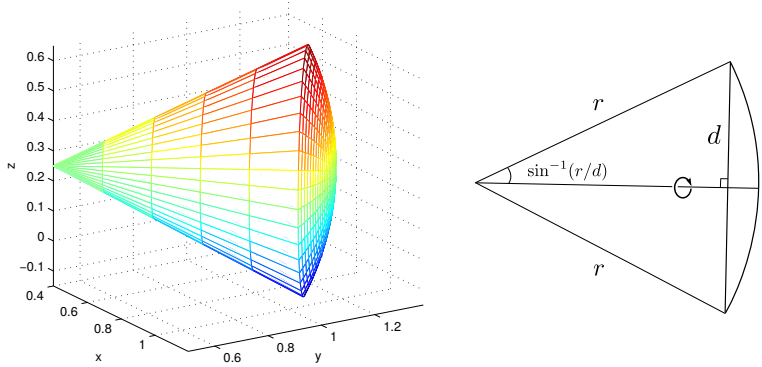}
\par\end{centering}

\caption{\label{fig:3D-sensor}The coverage area of a 3-dimensional sensor
is described by a spherical sector. The location of the sensor is
at $(0.4,0.5,0.25)$ which is the apex point of a spherical sector.
This sensor has finite sensing distance $r=0.4$ and a base radius
$d=0.8$.}
\end{figure}

Extension to 3D can be done by following \cite{TCOBS:JCP2004}. %which considers a real city model.
First, the level set
formulation introduced in Section \ref{sec:Coverage-optimization-GD}
is converted to the 3D setup. The coverage function $\phi(y;x,r,v)$
is computed by the 3D version of Algorithm \ref{PDE-based-algorithm}.
%This extension is already mentioned in \cite{TCOBS:JCP2004}. 
Then,
the coverage area of multiple sensors is computed by the integration
over the 3D computational domain. Secondly, we maximize the coverage
area with respect to four degrees of freedom for each sensor: two for adjusting
the viewing direction, and the other two for the coordinates on the boundary of the
3D obstacles. The optimization using the intermittent diffusion automatically approximates
one of the globally optimal positions of the sensor in terms of four degrees
of freedom. 

A 3D version of the example in Figure \ref{fig:corner} will be constructed on a 3D computational domain whose size is $[0,2]\times[0,2]\times[0,0.5]$.
We place two obstacles: a square box with the dimension $[0,1.4]\times[0,1.4]\times[0,0.5]$ and an upside down letter "L" shaped box with the same height. 
In Figure \ref{fig:3D_corner}(a), the left subfigure represents an initial position using a projection onto $x-y$ plane. The numbers besides the red crosses refer to the $z$-coordinate of the position. The middle subfigure shows the non-monitored regions 
with respect to the initial position of the sensors when we look at the environment in the front. The obstacles are shaded transparent gray, the blue surface represents the boundary of uncovered regions by the sensors, and the red surface is the cut due to the top $z=0.5$. 
While our algorithm changes the initial position as the subfigure (b) with 20 iterations, the 
expected coverage area increases from 0.1011 to 0.6259. It is also evident from the results that most of the areas are now covered under the new position of the sensors.

\begin{figure}[h]
\begin{centering}
\subfloat[The initial expected coverage area is 0.1011. 
(Left) Projection of the sensor locations onto $x-y$
plane. Red crosses denote the position of the sensor and the numbers shows the height of each sensor. 
%refer to the $z$-coordinate of the position
(Middle) A front side view in 3D. The obstacles are shaded transparent gray, the blue surface represents the boundary of uncovered regions by the sensors, and the red surface is the cut due to the top $z=0.5$. 
(Right) A reverse side view. ]{\begin{centering}

\includegraphics[width=3.7cm]{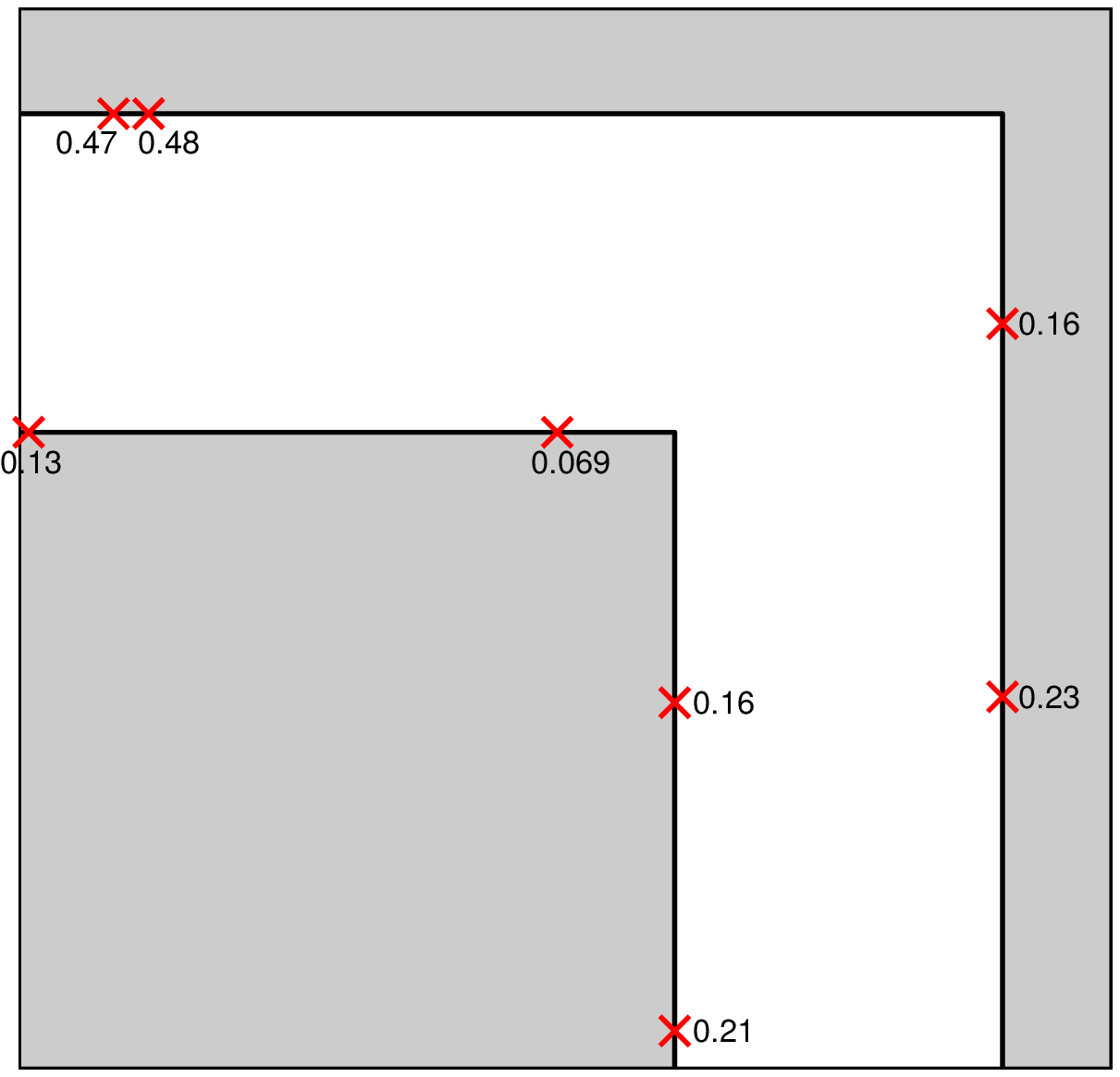}
\includegraphics[width=5cm]{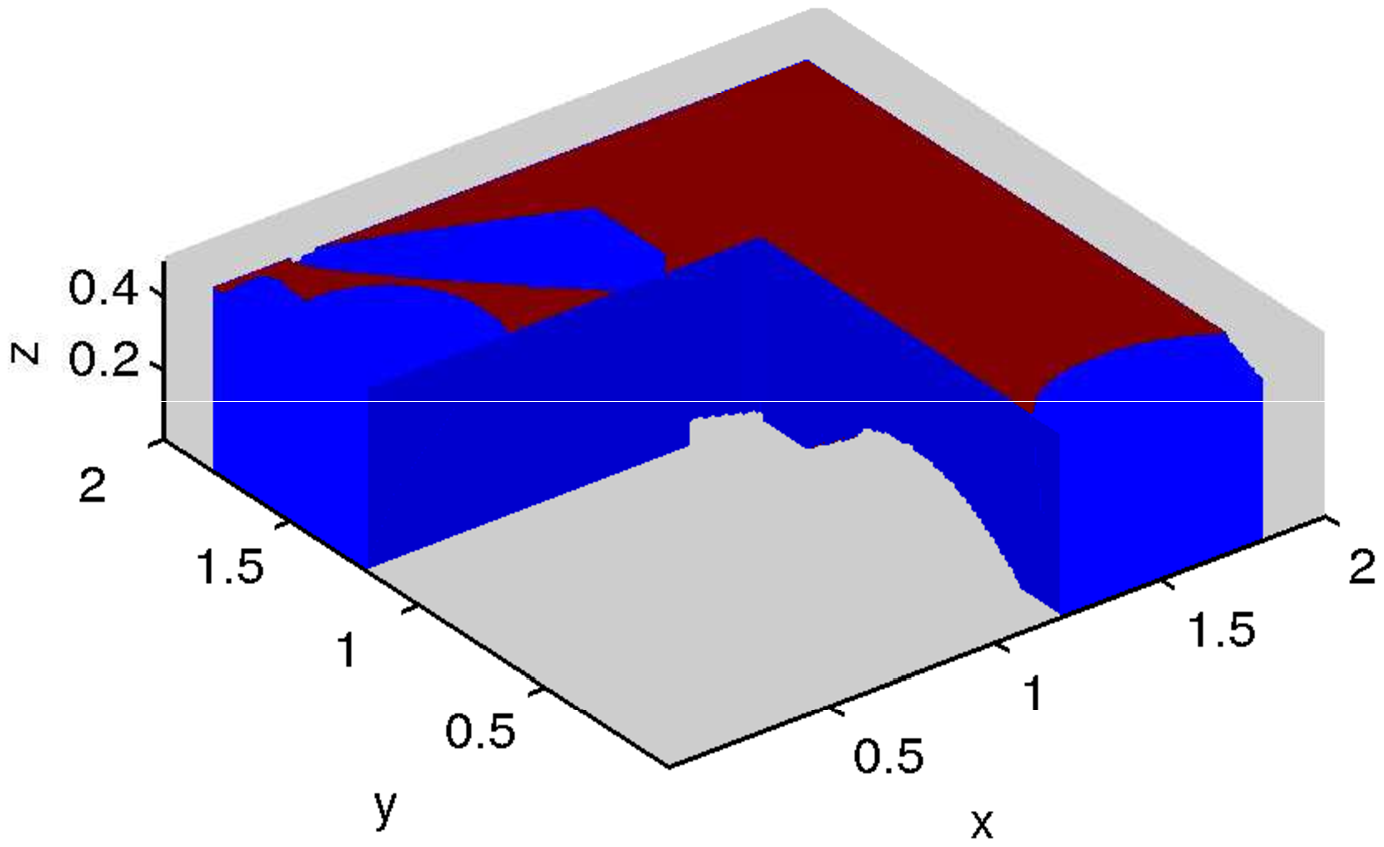}\includegraphics[width=5cm]{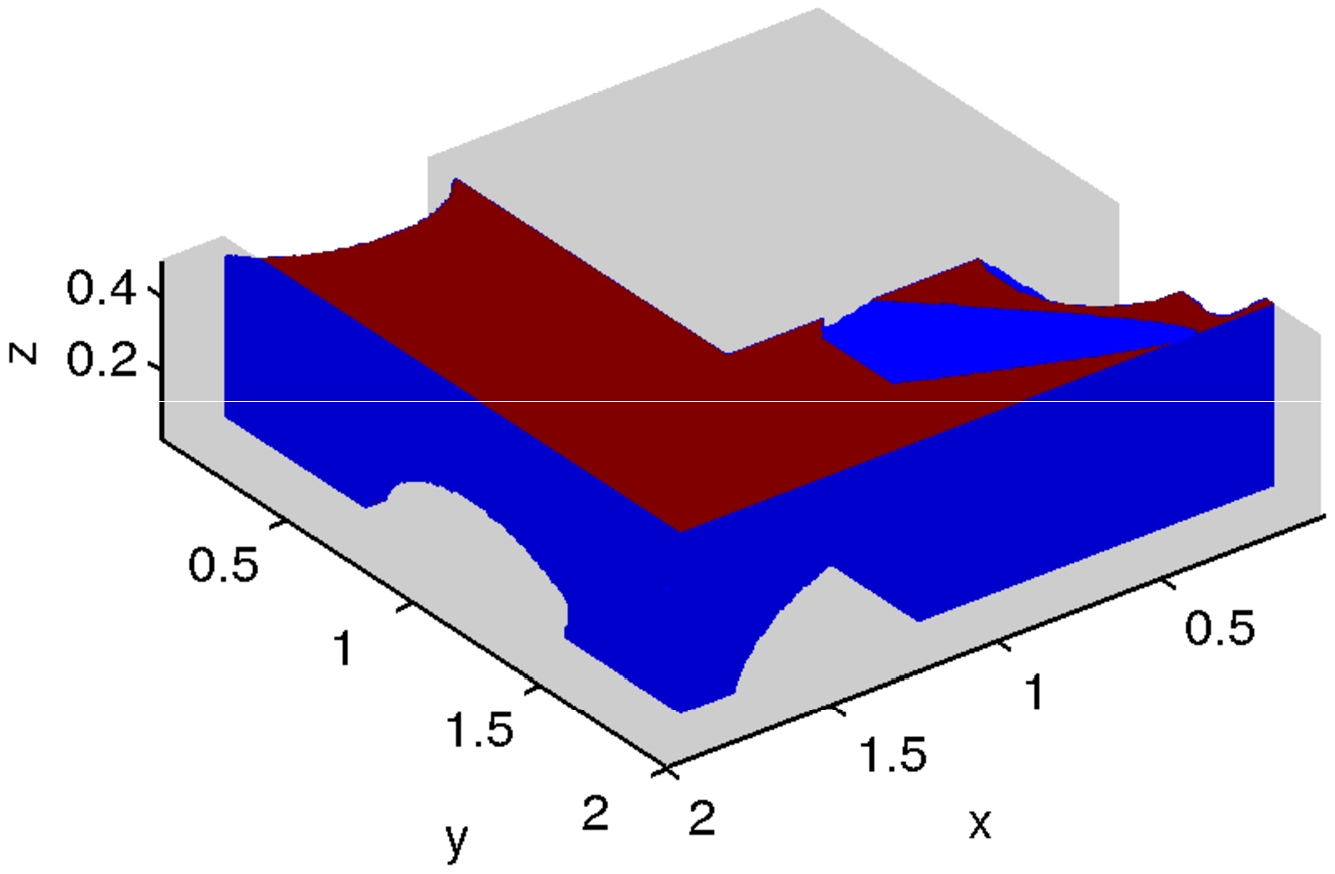}
\par\end{centering}

}
\par\end{centering}

\begin{centering}
\subfloat[After 20 iterations, the resulting expected coverage area is 0.6259.
(Left) Projection of the sensor locations onto $x-y$
plane. Red crosses denote the position of the sensor and the numbers shows the height of each sensor. 
(Middle) A front side view in 3D.
(Right) A reverse side view. ]{\begin{centering}

\includegraphics[width=3.7cm]{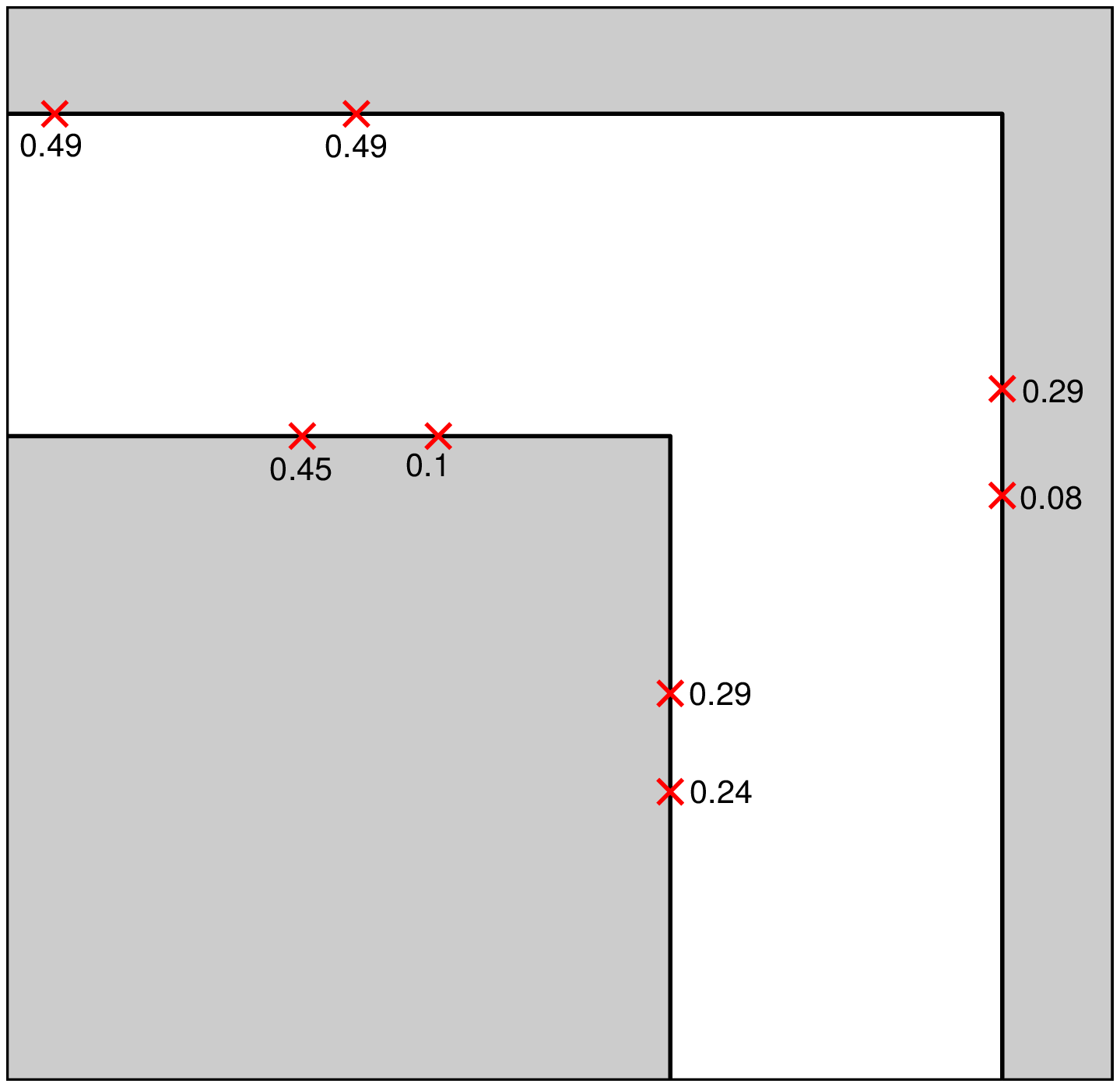}
\includegraphics[width=5cm]{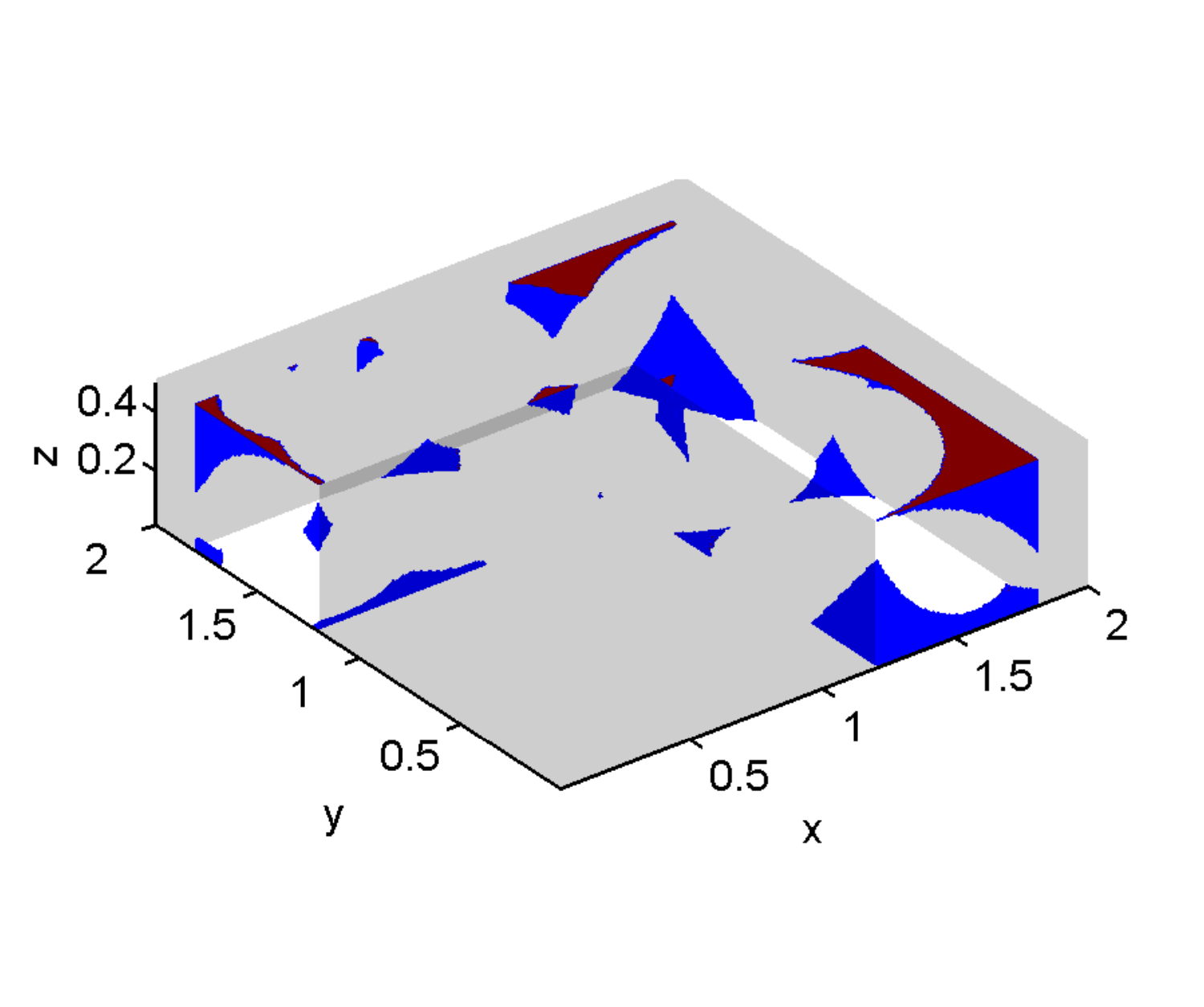}\includegraphics[width=5cm]{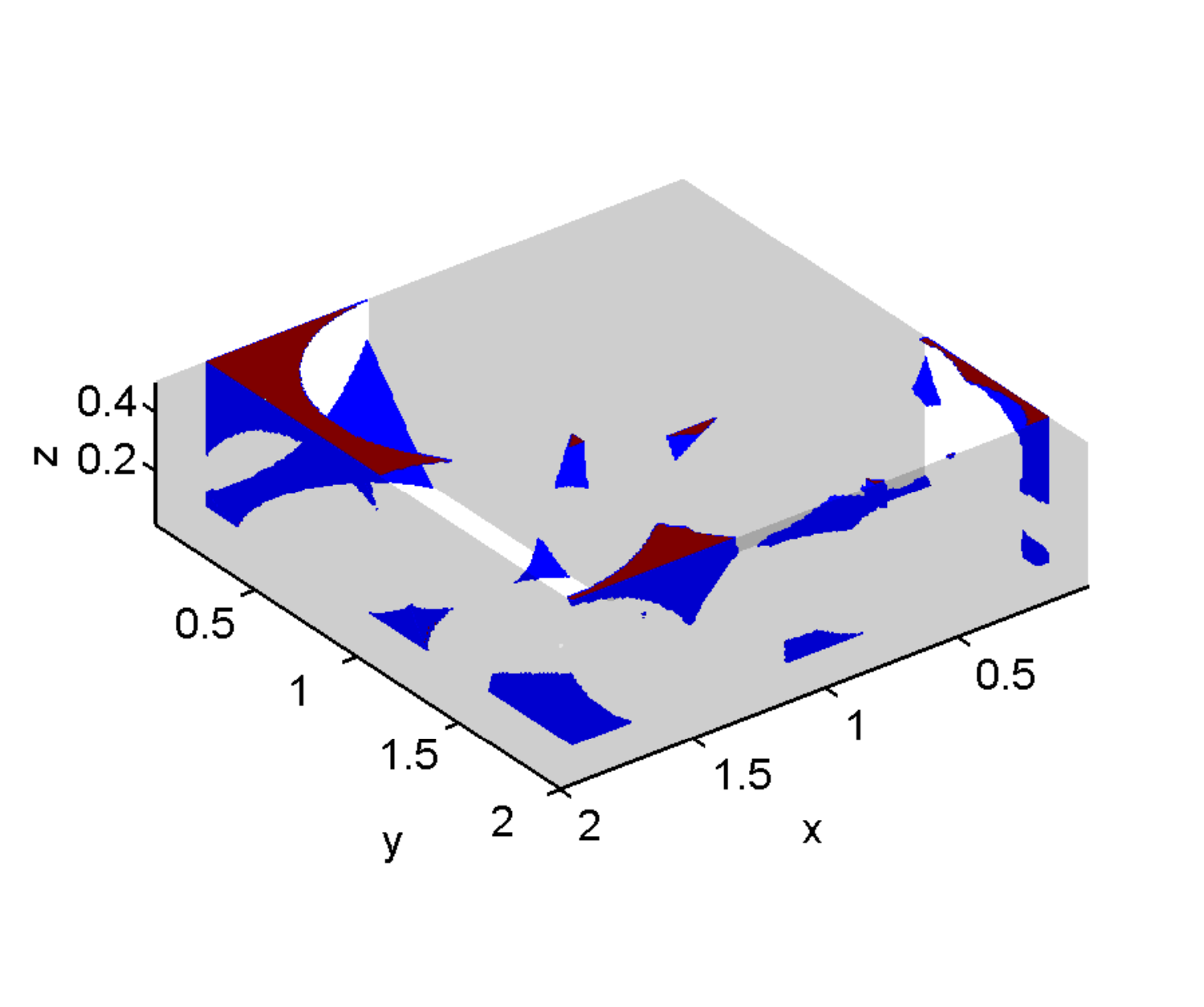}
\par\end{centering}

}
\par\end{centering}

\caption{\label{fig:3D_corner} A 3D alley with a corner. The resulting sensor positions are obtained
with 8 sensors using the parameters: sensing range $r=0.9$, base
radius $d=0.7$, step size $h=0.01$, and failure rate $p=0.5$. 
The result shows the movement of the sensors and how most of the areas are covered.
Our algorithm increases
the expected coverage area from 0.1011 to 0.6259.}
\end{figure}

\section{Summary\label{sec:Summary}}
We proposed a novel method for finding the optimal position of sensors
with a limited coverage range and width of viewing angle in a known environment. The
efficient positioning of such equipment is increasingly important
as it directly impacts the efficiency of allocated resources and
system performance. We considered two global optimization problems.
One is to achieve the greatest surveillance area of the region with
the given number of multiple sensors which have characteristics such
as range and angle limits. The other is to optimize the sensor placement
given that each sensor may randomly fails to operate.
In this case, the sensor placement maximizing the expected coverage
area in the environment will be the desired optimal solution. From
the results of several examples, we verified the effectiveness of
our algorithm. The more realistic scenario where sensor coverages are correlated
to each other shall be addressed in a future work.

\selectlanguage{american}%
\bibliographystyle{plain}
\bibliography{image_ref}
\selectlanguage{english}%

\end{document}